\documentclass[10pt]{amsart}
\usepackage{pinlabel}
\usepackage{amsmath, amsthm, amssymb, amscd, mathrsfs, eucal, epsfig, color, graphicx, subcaption}
\usepackage{hyperref}
\usepackage[margin=1in]{geometry}

\usepackage{url}
\hypersetup{breaklinks=true}
\usepackage{setspace, mathtools}

\usepackage{rotating}

\usepackage{tikz-cd}

\begin{document}

\newtheorem{theorem}{Theorem}[section]
\newtheorem{lemma}[theorem]{Lemma}
\newtheorem{corollary}[theorem]{Corollary}
\newtheorem{conjecture}[theorem]{Conjecture}
\newtheorem{proposition}[theorem]{Proposition}
\newtheorem{question}[theorem]{Question}
\newtheorem*{answer}{Answer}
\newtheorem{problem}[theorem]{Problem}
\newtheorem*{claim}{Claim}
\newtheorem*{criterion}{Criterion}

\newtheorem*{fact}{Fact}

\theoremstyle{definition}
\newtheorem{definition}[theorem]{Definition}
\newtheorem{construction}[theorem]{Construction}
\newtheorem{notation}[theorem]{Notation}
\newtheorem{convention}[theorem]{Convention}
\newtheorem*{warning}{Warning}
\newtheorem*{assumption}{Simplifying Assumptions}

\theoremstyle{remark}
\newtheorem{remark}[theorem]{Remark}
\newtheorem{example}[theorem]{Example}
\newtheorem{scholium}[theorem]{Scholium}
\newtheorem*{case}{Case}

\newcommand\tn[1]{\textnormal{#1}}

\newcommand\Z{\mathbb Z}
\newcommand\Q{\mathbb Q}
\newcommand\R{\mathbb R}
\newcommand\C{\mathbb C}
\renewcommand\H{\mathbb H}
\renewcommand\SS{\mathbb S}
\newcommand\E{\mathbb E}
\newcommand\SL{\textnormal{SL}}
\newcommand\PSL{\textnormal{PSL}}
\newcommand\GL{\textnormal{GL}}

\newcommand\Hom{\textnormal{Hom}}

\newcommand\into{\hookrightarrow}
\newcommand\onto{\twoheadrightarrow}

\title{Taut sutured handlebodies as twisted homology products}
\author[Margaret Nichols]{Margaret Nichols}
\address{Department of Mathematics \\ University of Chicago \\
Chicago, Illinois, 60637}
\email{mnichols@math.uchicago.edu}

\begin{abstract}
Friedl and Kim show any taut sutured manifold can be realized as a twisted homology product, but their proof gives no practical description of how complicated the realizing representation needs to be. We give a number of results illustrating the relationship between the topology of a taut sutured handlebody and the complexity of a representation realizing it as a homology product.
\end{abstract}

\maketitle

\section{Introduction}\label{s:intro}

A {\em sutured} $3$-manifold $(M,\gamma)$ is a manifold with boundary marked by a set of sutures, $\gamma$, which consists of oriented curves dividing $\partial M$ into oriented collections of components $R_+$ and $R_-$.\footnote{We note Gabai's original definition allowed sutures to consist of entire torus components of the boundary. Here we are interested in sutured handlebodies, and this aspect of the definition never arises.}

Gabai~\cite{gabai} introduced the notion of a \emph{taut} sutured manifold $(M,\gamma)$, which, roughly speaking, requires $M$ to be irreducible and the boundary components $R_\pm$ to be of minimal complexity.

Under suitable hypotheses, a sutured manifold is taut if $R_+$ and $R_-$ realize the Thurston norm of their (common) homology class. Here is an important example. Suppose $K$ is a knot in $S^3$, and let $R$ be a Seifert surface for $K$. Cutting $S^3$ open along $R$ produces a sutured manifold $M$ whose boundary decomposes along the knot $K$ into two copies $R_\pm$ of $R$. This
sutured manifold is taut precisely when $R$ is of minimal genus. Thus the theory of sutured manifolds can be (and is) used to compute knot genus.

Suppose $M$ is a sutured manifold, and $\alpha:\pi_1(M) \to \tn{GL}(V)$ is a representation. Then $\alpha$ restricts to representations $\pi_1(R_\pm) \to \tn{GL}(V)$, and we can define the twisted homology $H_*(M;E_\alpha)$, $H_*(R_\pm;E_\alpha)$. We say
that $M$ is an {\em $\alpha$-homology product} if the maps
$H_*(R_\pm,E_\alpha) \to H_*(M,E_\alpha)$ induced by inclusion are all isomorphisms. 
If $\alpha$ is not specified, we say $M$ is a {\em twisted homology product}.

This concept is important, because of

\begin{theorem}[Friedl-Kim \cite{FK13}]
If $M$ is a twisted homology product, it is taut.
\end{theorem}

Conversely, using Agol's Virtual Fibering Theorem (\cite{Agol}), they show

\begin{theorem}[Friedl-Kim \cite{FK13}]
If M is taut, it is a twisted homology product for some $\alpha$.
\end{theorem}

We call such an $\alpha$ {\em certifying} for $M$. The result of Friedl and Kim is not effective, in the sense that it gives no upper or lower bounds for the complexity of a certifying representation. This potentially reduces the practical value of twisted homology as a tool. Therefore, the fundamental question we study in this paper addresses precisely this issue:

\begin{question}
If $M$ is a taut sutured manifold, what is the simplest representation for which it is a twisted homology product, and what is the relationship of the complexity of the representation to the topology of $M$?
\end{question}

For $M$ a hyperbolic manifold, Agol and Dunfield found substantial computer evidence that $M$ is a twisted homology product for the geometric representation $\pi_1(M) \to \tn{SL}_2(\C)$ (\cite{AD}). They conjectured in general that every taut $M$ has a 2-dimensional certifying representation, and proved this for a simple class of manifolds, namely books of $I$-bundles.

For a given $M$ the search for a certifying representation falls into two parts: understanding the linear representations of $\pi_1(M)$, and understanding when such a representation is certifying. To simplify the discussion we restrict attention to the case that $M$ is a handlebody, so that $\pi_1(M)$ is free.

This case is of practical importance, since it often happens that the complement of a minimal genus Seifert surface is a handlebody.

\subsection{Statement of Results}

Our first result classifies the topologically simplest case, where $M$ is genus two, and both $R_\pm$ are once-punctured tori.

\begin{theorem}\label{Thm:ss}
Let $M$ be a taut sutured genus-two handlebody with a single connected suture. Then $M$ is a rational homology product.
\end{theorem}

In this setting, working with homology with twisted coefficients is unnecessary. In contrast, without the assumption of a single suture, the above is no longer true.

\begin{theorem}\label{Thm:nonhom}
There exists a taut sutured genus-two handlebody which is not a rational homology product.
\end{theorem}

The example in Theorem~\ref{Thm:nonhom} is an $\alpha$-homology product for a non-empty Zariski open set of choices $\alpha: \pi_1(M)\to \GL_1(\C)$. We conjecture this is the case in general for genus-two handlebodies.

\begin{conjecture}\label{Conj:g2}
Let $M$ be a taut sutured genus-two handlebody. The representations $\alpha: \pi_1(M)\to\GL_1(\C)$ which certify $M$ as an $\alpha$-homology product form a non-empty, Zariski open subset of the one-dimensional representation variety $\Hom(\pi_1(M), \GL_1(\C))$.
\end{conjecture}

This conjecture reflects the constricted nature of the genus-two setting. In contrast, as soon as we consider higher genus handlebodies, we can construct taut examples which require a two-dimensional representation.

\begin{theorem}\label{Thm:hg}
For all $g\ge 3$, there are taut sutured handlebodies $M_g$ of genus $g$ which fail to be a twisted homology product for any one-dimensional representation.
\end{theorem}

To construct these examples, we describe a condition on how $\pi_1(R_\pm)$ sit inside $\pi_1(M)$ which prevents $M$ from being a one-dimensional twisted homology product.

Finally, we generalize this to provide obstructions for admitting solvable representations of arbitrarily large derived length. In particular, we are able to prove the following strong negation of Agol and Dunfield's conjecture in the restricted setting of solvable representations.

\begin{theorem}\label{Thm:dim}
There exist taut manifolds $M_k$ such that $M_k$ is not a twisted homology product for any solvable representation $\alpha:\pi_1(M_k)\to \tn{GL}_{\varphi(k)}(\C)$, where $\varphi(k)\to\infty$ with $k$.
\end{theorem}

The paper is organized as follows. In Section~\ref{s:defs}, we briefly review the theory of taut sutured manifolds and taut sutured manifolds. In Section~\ref{s:bc}, we review the basic of commutator calculus, which we use in Section~\ref{s:ss} to prove Theorem~\ref{Thm:ss}. In Section~\ref{s:ss} we also introduce the example illustrating Theorem~\ref{Thm:nonhom}, which we address again in Section~\ref{s:1d}. Section~\ref{s:1d} addresses the situation of one-dimensional representations and describes specific conditions for being a one-dimensional twisted homology product. We use these conditions in Section~\ref{s:hg} to prove Theorem~\ref{Thm:hg}. Finally, in Section~\ref{s:solv} we generalize the results of Sections~\ref{s:1d} and \ref{s:hg} to prove Theorem~\ref{Thm:dim}.

\subsection*{Acknowledgements}
The author would like to thank Danny Calegari for his continuing support and guidance throughout this work. The author also thanks Nick Salter for a multitude of helpful conversations and numerous comments on drafts of this paper.

\section{Basic definitions and facts}\label{s:defs}

\subsection{Sutured manifolds}

\begin{definition}
A \emph{sutured manifold} is a four-tuple $(M,R_\pm,\gamma)$ consisting of a compact $3$-manifold $M$ and a collection of pairwise disjoint, embedded curves $\gamma\subset \partial M$, which partition $\partial M -\gamma$ into oriented subsurfaces $R_+$ and $R_-$, such that the orientations induced on their common boundary $\gamma$ agree.
\end{definition}

Though this definition does not require it, we will always assume $M$ is connected. Some sources define the sutures to be a collection of annuli; our definition as a collection of curves is equivalent, though we occasionally view the sutures as annuli when convenient for notational or conceptual purposes.

\begin{example}
\ \\ \vspace{-10pt}
\begin{enumerate}
\item Given any compact surface $S$, the manifold $M=S\times I$ can be given a natural sutured structure, where $\gamma = \partial S\times I$, $R_+=S\times 1$ and $R_-=S\times 0$.
\item Any Seifert surface $S$ associated to a knot $K$, or more generally a link $L$, defines a sutured manifold $S^3-N(S)$, with $\gamma=K$ (or $L$) and $R_\pm\cong S$. The knot (or link) is fibered by $S$ exactly when this sutured manifold is a product.
\end{enumerate}
\end{example}

We are particularly interested in \emph{taut} sutured manifolds, which we define below. We recall first the \emph{Thurston norm} on $H_2(M,\partial M)$. Given a connected embedded surface $(S,\partial S)\subseteq (M,\partial M)$, we define $\chi_-(S)=\max\{0, -\chi(S)\}$. For $S$ not connected, $\chi_-(S)=\sum_{T\subseteq S} \chi_-(T)$, taken over connected components of $S$. Finally, the Thurston norm of $\sigma\in H_2(M,\partial M)$ is defined as
\[
\|\sigma\|=\min_{[S]=\sigma}\chi_-(S).
\]

\begin{definition}
A sutured manifold $M$ is \emph{taut} if it is irreducible and $R_\pm$ are \emph{taut}, that is, they are incompressible and realize the Thurston norm of their homology class. It is \emph{balanced} if $M$ is irreducible and $\chi(R_+)=\chi(R_-)$, and moreover $M$ is not a solid torus without sutures, and if any component of $R_\pm$ has positive Euler characteristic, then $M$ is $D^3$ with a single suture.
\end{definition}

Notice that a taut sutured manifold is necessarily balanced.

\subsection{Twisted homology products}

Associated to any representation $\alpha:\pi_1(M)\to \GL(V)$ of the fundamental group of a sutured manifold $M$ are homology groups $H_*(M;E_\alpha)$ and cohomology groups $H^*(M;E_\alpha)$ with coefficients twisted by the representation $\alpha$. Any such representation restricts to representations $i^*_\pm\alpha:\pi_1(R_\pm)\to \GL(V)$, in turn giving natural maps
\[
H_*(R_\pm;E_\alpha)\xrightarrow{i_*^\pm} H_*(M;E_\alpha),
\]
and similarly on cohomology.

Recall that a sutured manifold $M$ is an {\em$\alpha$-homology product} for a representation $\alpha:\pi_1(M)\to \GL(V)$ if the maps $i^\pm_*$ are all isomorphisms. Equivalently, we require
\[
H_*(M,R_\pm; E_\alpha)=0.
\]

Our interest in twisted homology products is motivated by the following theorem of Friedl and Kim (\cite{FK13}).

%
%

\begin{theorem}[Friedl-Kim]
Let $M$ be a balanced sutured manifold. Then $M$ is taut if and only if $M$ is an $\alpha$-homology product for some $\alpha: \pi_1(M)\to \tn{GL}_n(\C)$.
\end{theorem}

In particular, the representation $\alpha$ may always be taken to be a unitary representation. This proves any taut sutured manifold can be realized as a twisted homology product, giving a novel method for verifying tautness of sutured manifolds. However, their construction of the certifying representation uses in a key way Agol's virtual fibering (\cite{Agol}). 

We will often be interested in representations which satisfy a homological generalization of the condition that $E_\alpha$ and $E_\alpha^*$, its dual, be isomorphic.

\begin{definition}
A representation $\alpha:\pi_1(M)\to \GL(V)$ is \emph{homologically self-dual} if, for any subspace $A\subseteq M$, there is an isomorphism $H_*(M, A; E_\alpha)\cong H^*(M, A; E_\alpha)$.
\end{definition}

For example, any unitary representation is homologically self-dual, as is any representation to $SL_2(K)$, for any field $K$. This condition is of particular use because it greatly simplifies verifying $M$ as a twisted homology product.

\begin{proposition}[Agol-Dunfield, Proposition 3.1]\label{Prop:AD3.1}
Suppose $M$ is a connected, balanced sutured manifold with $R_\pm$ nonempty. If $\alpha$ is homologically self-dual, then $M$ is an $\alpha$-homology product if and only if any one of the following vanish:
\[
H_k(M, R_\pm; E_\alpha)\tn{, }H^k(M,R_\pm; E_\alpha) \tn{ for }k=1,2.
\]
\end{proposition}

We give their proof to highlight a couple of facts which do not need the assumption of homological self-duality.

\begin{proof}
As $R_\pm$ are nonempty, we know $H_0(M,R_\pm;E_\alpha)=H^0(M,R_\pm;E_\alpha)=0$. By Poincar\'e duality, also $H_3(M,R_\mp;E_\alpha)=0$. Now suppose $H_1(M,R_-;E_\alpha)=0$; the other cases are similar. Since $M$ is balanced, we have $\chi(R_\pm)=\chi(M)$, so $\chi(H_*(M,R_-;E_\alpha))=0$. Then, since $H_k(M,R_-;E_\alpha)=0$ for $k\ne2$, we also have $H_2(M,R_-;E_\alpha))=0$. Poincar\'e duality now shows $H^*(M,R_+;E_\alpha)=0$. Finally, as $\alpha$ is homologically self-dual, this gives $H_*(M,R_+;E_\alpha)=H^*(M,R_+;E_\alpha)=0$. 
\end{proof}

We do not use self-duality until the last step. More generally, we can say

\begin{corollary}
For $R=R_\pm$,
\[
H_1(M, R; E_\alpha)=0 \quad \iff \quad H_2(M, R; E_\alpha)=0,
\]
and 
\[
H^1(M, R; E_\alpha)=0 \quad \iff \quad H^2(M, R; E_\alpha)=0.
\]
\end{corollary}

\begin{corollary}
$M$ is an $\alpha$-homology product if and only if
\[
H_1(M, R_+; E_\alpha)=H_1(M, R_+; E_\alpha^*)=0.
\]
In particular, if $M$ is an $\alpha$-homology product, it is also an $\alpha^*$-homology product.
\end{corollary}

\subsection{Sutured manifold hierarchies}

To conclude this section, we discuss one method we might try to use for constructing representations, and why it fails. Recall the sutured manifold hierachy of a taut sutured manifold $M$ is a sequence of decompositions
\[
M = M_0 \xrightarrow{S_1} M_1 \xrightarrow{S_2} M_2 \xrightarrow{S_3} \cdots \xrightarrow{S_n} M_n
\]
such that each $S_k$ meets the sutures of $M_{k-1}$ transversally, each $M_k$ is taut, and every embedded surface in $M_n$ is separating. Gabai introduced this concept in \cite{gabai-hierarchies}, proving such hierarchies always exist, and moreover, that if a sequence of decompositions of an arbitrary sutured manifold $M$ satisfies certain additional conditions, tautness of $M_n$ implies $M$ is taut as well.

As these hierarchies are often used in inductive arguments, one might hope that such a hierarchy can be used to inductively construct certifying representations. More precisely, if $M \xrightarrow{S} N$ is a decomposition, then $N$ is a subspace of $M$, so a representation of $M$ restricts to a representation of $N$. Suppose $M$ and $N$ are both taut, and that $\alpha$ is certifying for $M$. One might na\"ively imagine that the restriction of $\alpha$ is certifying for $N$. This is not true, as the following example shows.

\begin{example}\label{ex:smh}
The handlebodies $M$ and $N$ in Figure~\ref{Fig:hierex} are related by a decomposition along a disk meeting the sutures in $M$ in four points. In this case, we may realize $\pi_1(M)$ as an HNN extension of $\pi_1(N)\cong F_2$, with $\pi_1(M)\cong F_3$ gaining a free generator $z$. The representation $\alpha:\pi_1(M)\to \tn{GL}(\C)$ defined by $\alpha(x)=-1$ and $\alpha(y)=\alpha(z)=1$ is certifying for $M$, as can be verified via Proposition~\ref{Prop:tautness}. However, when restricted to $N$, the representation $\alpha$ is no longer certifying: the locus of representations which fail to be certifying are those with $x\mapsto -1$.

\begin{figure}[h!]
  \includegraphics[width=\linewidth]{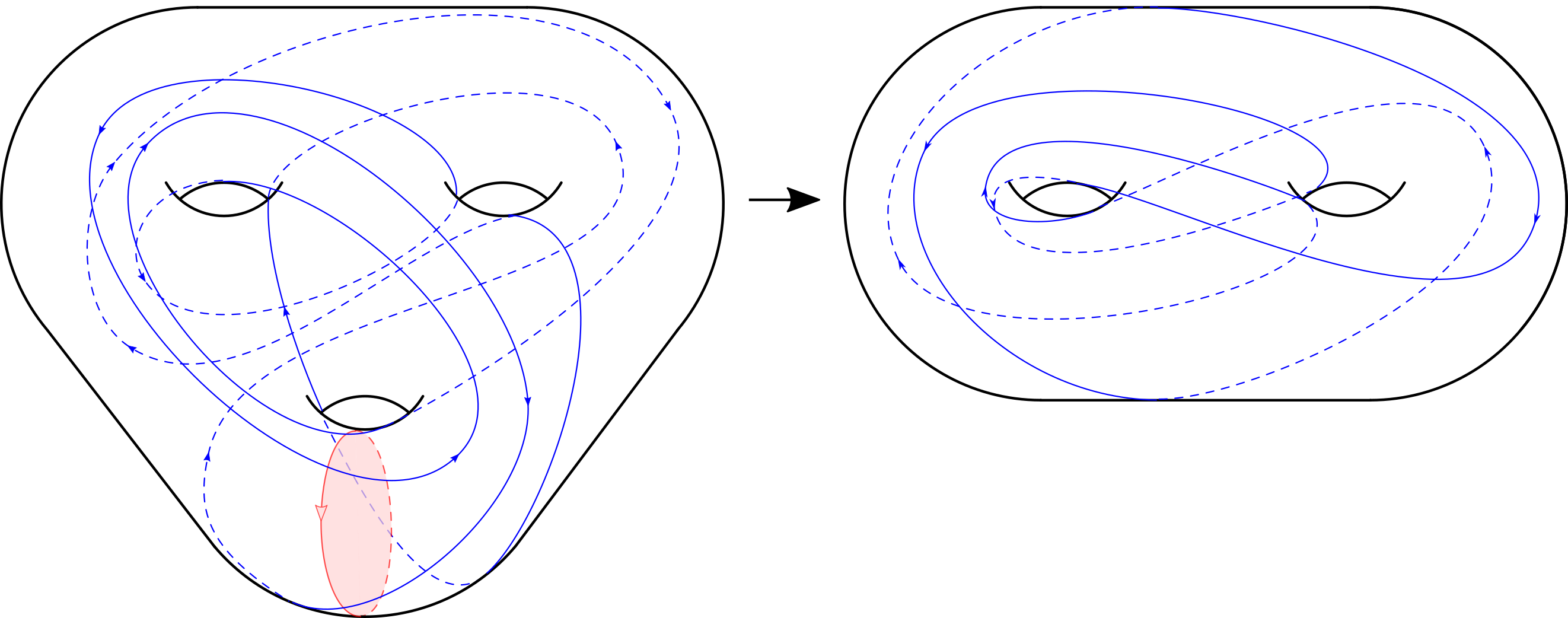}
  \caption{The decomposition of $M$ (left) along the disk $S$ to obtain $N$ (right).}
  \label{Fig:hierex}
\end{figure}
\end{example}
 
The reason for this is that there is part of the boundary of $N$ which is not contained in the boundary of $M$. Understanding when this naive guess fails requires analyzing how the suture structure changes with this new boundary, which is subtle in practice. However, this failure is isolated to the local situation of the decomposition. That is, if $S_\pm\subseteq N$ are the two copies of $S$ in the boundary of $N$, there is still an injection $H_*(R_\pm-S_\pm; E_{\alpha|_N})\into H_*(N; E_{\alpha|_N})$.


In this example, it is the case that both manifolds admit one-dimensional certifying representations. However, even the condition for admitting a one-dimensional certifying representation is subtle to understand in relation to a decomposition $M \xrightarrow{S} N$.

As we will see in Lemma~\ref{Lem:addhandle}, in the special case that the surface $S$ is a disk meeting the sutures of $M$ exactly twice, a certifying representation for $N$ can be extended to one which certifies $M$.


\section{The basics of basic commutators}\label{s:bc}

Before we prove Theorem~\ref{Thm:ss}, we first review the theory of {\em basic commutators}, due to Hall~(\cite{hall}).

Recall the {\em lower central series} $G_k$ of a group $G$ is defined as $G_0=G$ and $G_{k}=[G,G_{k-1}]$ for $k>0$. A group is nilpotent exactly when this series has finite length. No free group $F_n$ of rank $n\ge 2$ is nilpotent; however, it is well-known that $F_n$ is {\em residually nilpotent}, that is,
\[
\bigcap_{k=0}^\infty G_k=1.
\]

Now assume $G\cong F_n$ is free of rank $n$. Taking quotients of successive terms of the series gives a sequence of free abelian groups $G_{k-1}/G_k$. Hall defined basic commutators, and proved the basic commutators of weight $k$ form a generating set for the corresponding quotient. The basic commutators are defined inductively, with respect to a fixed, ordered generating set $\{x_1,\ldots,x_n\}$, equipped with an inductively defined weighting and ordering. Specifically:

\begin{enumerate}
\item The basic commutators of weight 1 are $x_1,\ldots,x_n$, and ordered by $x_i\le x_j$ iff $i\le j$;
\item The basic commutators of weight $k>1$ consist of words $[x,y]$, where $x,y\in G$ are of weights $i,j$ respectively, such that
	\begin{enumerate}
		\item[(i)] $k=i+j$;
		\item[(ii)] $x<y$ according to the ordering; and 
		\item[(iii)] if $y=[w,z]$, then $x\ge w$;
	\end{enumerate}
\item The basic commutators of weight $k$ are then given an (arbitrary) order, and set to be greater than all basic commutators of lesser weight.
\end{enumerate}

\begin{example}
If $G=F(x_1,x_2)$ is a free group on two generators, then the basic commutators of weight $k$ are shown below for small $k$. 
\begin{align*}
k=1: & \quad x_1, x_2 \\
2: & \quad [x_1, x_2] \\
3: & \quad [x_1,[x_1,x_2]], [x_2,[x_1,x_2]] \\
4: & \quad [x_1,[x_1,[x_1,x_2]]], [x_1,[x_2,[x_1,x_2]]], [x_2,[x_1,[x_1,x_2]]], [x_2,[x_2,[x_1,x_2]]]
\end{align*}
Note that $k=5$ gives the first example of a basic commutator $[x,y]$ where the weight of $x$ is not $1$: $[[x_1,x_2],[x_1,[x_1,x_2]]]$.
\end{example}

\begin{proposition}\label{Hall-coll}
Any element $g$ in a free group $G$ can be uniquely expressed in the form
\[
g = c_1 c_2 \cdots c_m,
\]
where each $c_i$ is a basic commutator or its inverse, $c_i \le c_j$ whenever $i < j$, and $c_i\ne c_{i+1}^{-1}$.
\end{proposition}

\begin{lemma}\label{Lem:comm}
For a fixed $\gamma\in G - G_1$, for all $k\ge 1$, commutation with $\gamma$ defines a homomorphism
\[
[\gamma,\cdot\,]: G_{k-1} \to G_k/G_{k+1}.
\]
Moreover, the image of $[\gamma,\cdot\,]$ consists of all $g\in G_k/G_{k+1}$ which can be expressed as $g=[\gamma,g']$ for some $g'\in G$.
\end{lemma}

Note in particular, this Lemma tells us that even given $g=[\gamma, g']$ where $g'\not\in G_{k-1}$, as long as $g\in G_k$, we can find some $h\in G_{k-1}$ with $g=[\gamma, h]$.

\begin{proof}
To see $[\gamma,\cdot\,]$ is a homomorphism, we observe that for any $g,h\in G_{k-1}$,
\begin{align}\label{Eqn:prod}
[\gamma, gh] = [\gamma, g][\gamma, h][[h,\gamma], g] \equiv [\gamma, g][\gamma, h] \mod G_{k+1}.
\end{align}

Now we wish to show that for any $[\gamma,g]\in G_k$, in fact $[\gamma,g] \equiv [\gamma,h] \mod G_{k+1}$ for some $h\in G_{k-1}$. Express $g= c_1 c_2\cdots c_m$ as in Proposition \ref{Hall-coll}. Then,
\[
[\gamma, g] = [\gamma, c_1 c_2 \cdots c_m].
\]
If $c_{m_1}\cdots c_{m_2}$ is the subword of commutators of weight ${k}$, we will show that we can take $h = c_{m_1}\cdots c_{m_2}$. In particular, it follows that if the commutator expression of $g$ consists of commutators of weight at most $k-1$, then $[\gamma, g]$ cannot lie in $G_k$ unless it is trivial.

We proceed by induction on $k$. In the case that $k=1$, then $G_{k-1}=G$, and the Lemma already holds. To see we can choose $h$ of the desired form, observe that if the first $m_2$ commutators $c_i$ are of weight $0$, then by (\ref{Eqn:prod}),
\[
[\gamma, g] \equiv [\gamma, c_1 \cdots c_{m_2}][\gamma, c_{m_2+1} \cdots c_m] \equiv [\gamma, c_1 \cdots c_{m_2}] \mod G_2.
\]

Suppose then $k>1$. We first address the case where $\gamma$ is a generator of $G$. As above, by (\ref{Eqn:prod}),
\[
[\gamma, g] \equiv [\gamma, c_1 \cdots c_{m_2}]  \mod G_{k+1}.
\]
Moreover, noting that (\ref{Eqn:prod}) still holds when just $h\in G_{k-1}$, we have
\[
[\gamma, g] \equiv [\gamma, c_1 \cdots c_{m_1-1}][\gamma, c_{m_1} \cdots c_{m_2}] \mod G_{k+1}.
\]
This shows $[\gamma, c_1 \cdots c_{m_1-1}]$ must also lie in $G_k$. That is to say,
\[
[\gamma, c_1 \cdots c_{m_1-1}] \equiv 1 \mod G_k.
\]
Let $\ell$ denote the largest index of a commutator of weight less than $k-1$. Then, again,
\[
[\gamma, c_1 c_2 \cdots c_{m_1-1}] \equiv [\gamma, c_1 c_2 \cdots c_{\ell}] [\gamma, c_{\ell+1}\cdots c_{m_1-1}] \equiv 1 \mod G_k.
\]
As $[\gamma, c_{\ell+1}\cdots c_{m_1-1}] \in G_{k-1}$, so too $[\gamma, c_1 c_2 \cdots c_{\ell}]\in G_{k-1}$. But by induction, then $[\gamma, c_1 c_2 \cdots c_{\ell}]$ must be trivial. Thus in fact $[\gamma, c_{\ell+1}\cdots c_{m_1-1}]\equiv 1 \mod G_k$.

Now, if we chose our ordering to have $\gamma$ as the smallest generator, then
\[
[\gamma, c_{\ell+1}\cdots c_{m_1-1}] \equiv \prod [\gamma, c_i] \mod G_k
\]
is a product of basic commutators of weight $k$. As these basic commutators freely generate $G_{k-1}/G_k$, each $[\gamma, c_i]$ must be trivial, which is to say, each $c_i$ must be trivial. Hence $[\gamma, c_1\cdots c_{m_1-1}]=1$, and $[\gamma, g] \equiv [\gamma, c_{m_1}\cdots c_{m_2}] \mod G_{k+1}$, as desired.

%
%
%
%

Suppose now $\gamma$ is not a generator, but still $\gamma\not\in G_1$. Since $k>1$, $h=c_{m_1}\cdots c_{m_2}$ does not depend on the ordering of the generators $x_i$. Thus the argument above shows for all $i$,
\[
[x_i,h]\equiv [x_i,g] \mod G_{k+1}.
\]
Hence each $[x_i,g]\in G_k$. Returning one last time to (\ref{Eqn:prod}), note that the corresponding product formula holds for the first entry, and moreover in both cases, it suffices to know $[\gamma, g],[\gamma, h]\in G_k$.

So finally,
\[
[\gamma, h] \equiv \prod [x_i, h] \equiv \prod [x_i, g] \equiv [\gamma, g] \mod G_{k+1}. \qedhere
\]
\end{proof}

For elements $\gamma,\delta\in G$, we write $\Gamma_k=\tn{Im}([\gamma,\cdot\,]: G_{k-1} \to G_k/G_{k+1})$, and similarly $\Delta_k=\tn{Im}([\delta,\cdot\,]:G_{k-1} \to G_k/G_{k+1})$.

\begin{lemma}\label{Lem:comm2}
Let $k>1$ and $\gamma,\delta\in G$ such that $\gamma\not \equiv \delta\not \equiv 1 \mod G_1$. Then $\Gamma_k \cap \Delta_k=1$.
\end{lemma}

\begin{proof}
We need to check that $[\gamma, g] \not\equiv [\delta, h]$ for any $g,h\in G_{k-1}$. We first note that $\Gamma_k$ (resp.\ $\Delta_k$) is freely generated by elements $[\gamma, c]$ (resp.\ $[\delta,c]$) where $c$ is a basic commutator of weight $k$. This follows from Lemma~\ref{Lem:comm} and the observation that $[\gamma,c]\ne[\gamma, c']$ if $c$ and $c'$ are distinct basic commutators (and similarly for $\delta$). In particular, we can factor
\[
[\gamma, c] \equiv \prod [x_i,c]^{k_i} \mod G_{k+1},
\]
and since $k>1$, each $[x_i,c]$ is itself a basic commutator. So it is enough to show that these generating sets are disjoint.

Suppose not, say, $[\gamma, c] = [\delta, c']$ for basic commutators $c$, $c'$. Then, factoring, we have
\[
\prod [x_i, c]^{k_i} \equiv \prod [x_i, c']^{\ell_i} \mod G_{k+1}.
\]
Thus $c=c'$ and $k_i=\ell_i$ for all $i$. But then
\[
[\gamma\delta^{-1}, c] \equiv [\gamma, c][\delta^{-1}, c] \equiv \prod [x_i, c]^{k_i}\prod[x_i,c]^{-k_i} = 1 \mod G_{k+1}.
\]
As $c$ is a basic commutator of weight $k$, we must have $\gamma\delta^{-1}\in G_1$, which contradicts our assumption that $\gamma \not\equiv \delta \mod G_1$.
\end{proof}

\section{The case of a single suture in genus two}\label{s:ss}

We are now prepared to prove our first theorem.

\begin{theorem}[Theorem~\ref{Thm:ss}]
Let $M$ be a taut sutured handlebody of genus two with a single connected suture. Then $M$ is a rational homology product, i.e., $M$ is certified by the trivial representation.
\end{theorem}

\begin{proof}
Suppose $M$ is not a homology product. We will show this implies the image of $\gamma$ in $\pi_1M$ must be trivial, contradicting the condition that $\pi_1R_\pm$ inject.

Notice $R_\pm$ are necessarily once-punctured tori, so $M$ is balanced. By Proposition~\ref{Prop:AD3.1} as applied to the trivial representation, both $H_1(M, R_\pm)\ne 0$. In particular, by the exactness the long exact sequence of pairs
\[
\cdots \to H_1(M, R_\pm) \to H_1(R_\pm)\to H_1(M) \to \cdots,
\]
neither $H_1(R_\pm)$ surject onto $H_1(M)$.


Mayer-Vietoris gives the exact sequence
\[
H_1(\gamma)\to H_1(R_+)\oplus H_1(R_-)\to H_1(\partial M).
\]
The connected suture $\gamma$ is a boundary in both $R_\pm$, so the image of the first map is zero and by exactness, the second map is an injection. Working with rational coefficients, a dimension count shows it is in fact an isomorphism.

When $R_\pm$ are then included into $M$, the corresponding composition on homology, $H_1(R_+)\oplus H_1(R_-)\to H_1(M)$, is a surjection. As neither $H_1(R_\pm)$ individually surject, the image of each has rank $1$.



Now consider $\pi_1(R_\pm)$ as subgroups of $\pi_1(M)$. We may pick $a,b\in \pi_1(R_+)$, $c,d\in \pi_1(R_-)$ so that $\gamma=[a,b]=[c,d]$. Suppose $a$ and $c$ represent nonzero homology classes in $H_1(M)$. Then we may express $[b]$ and $[d]$ as rational multiples of $[a]$ and $[c]$, respectively; say, $p[a]=q[b]$ and $r[c]=s[d]$. Fix $b'=b^qa^{-p}$ and $d'=d^rc^{-s}$; then $[b'] = [d'] = 0 \in H_1(M)$.


In the notation of Lemma~\ref{Lem:comm2}, we see $[a,b']\in A_2$. We claim then that $[a,b]\in A_2$. It suffices, by Lemma~\ref{Lem:comm}, to show $[a,b]\in (\pi_1(M))_2$. Observe that
\[
[a,b'] \equiv [a,b]^q \mod (\pi_1(M))_2.
\]
Since $(\pi_1(M))_2$ is torsion free, our claim follows. Similarly, $[c,d]\in C_2$. We conclude $\gamma = [a,b]=[c,d]\in A_2\cap C_2$.

Now, $a$ and $c$ must have distinct images in $H_1(M)$, since together they span the image of $H_1(R_+)\oplus H_1(R_-)$ in $H_1(M$). Thus $A_2\cap C_2=1$ by Lemma~\ref{Lem:comm2}, so in fact $[a,b], [c,d]\in (\pi_1(M))_3$. Then $[a,b], [c,d]\in A_3\cap C_3=1$, and inducting in this way shows $[a,b], [c,d]\in (\pi_1(M))_k$ for all $k$. But as remarked earlier, the infinite intersection of these groups is trivial. Thus we must have had $\gamma=1\in \pi_1(M)$.
\end{proof}





\begin{remark}
In the situation of a single suture on a handlebody $M$ of genus two, note that $M$ is a rational homology product exactly when the suture curve $\gamma$ has nontrivial image in $\pi_1(M)$.
\end{remark}


\subsection{Failure with multiple sutures} When the assumption of a single suture is dropped, the conclusion of Theorem~\ref{Thm:ss} no longer holds. Consider the following example.

\begin{figure}[h!]
  \includegraphics{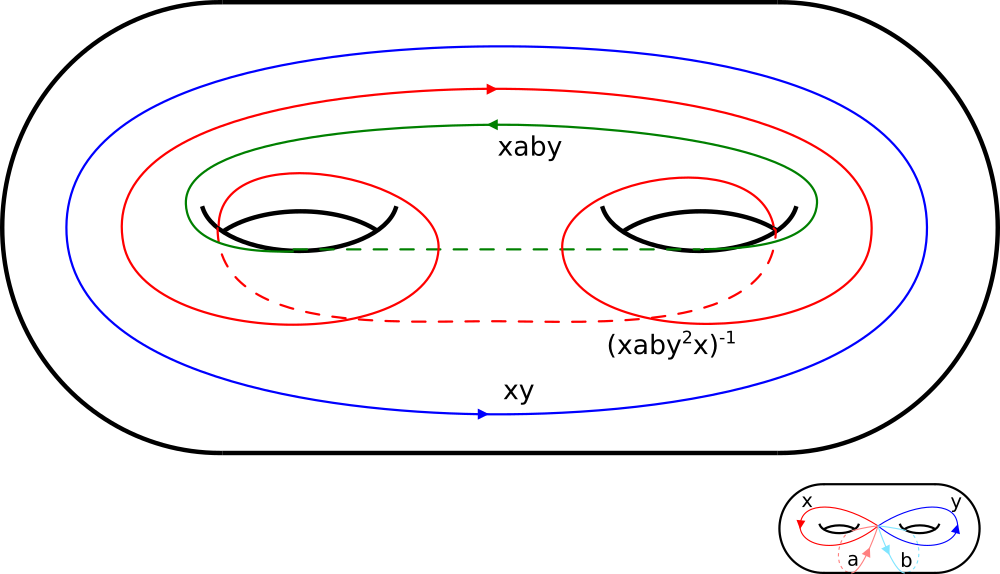}
  \caption{}
  \label{Fig:g2ex}
\end{figure}

\begin{example}[Theorem~\ref{Thm:nonhom}]\label{ex:gen2}
Let $M$ be a genus-two handlebody, with suture $\gamma$ consisting of the three curves shown in Figure~\ref{Fig:g2ex}. These correspond to the free homotopy classes $yx$, $xaby$, and $(xaby^2x)^{-1}$.

The boundary components $R_\pm$ are topological pants. Their fundamental groups, as subgroups of $\pi_1(\partial M)$, are both freely generated by $yx$ and $xaby$. These inject into $\pi_1(M)$ as the subgroup $\langle xy, yx\rangle$. Abelianizing, we see this is not a homology product: the generators of the fundamental group map to the same cycle in $H_1(M)$.
\end{example}

We return to this example in the next section to prove its tautness using tools developed therein.

\begin{remark}
Though multiple sutures are required in genus two, more generally a taut sutured manifold with a single suture can fail to be a rational homology product. We can construct a genus four handlebody with a single suture with this same property. To do so, we attach two {\em sutured one-handles} to $M$ above (see Section~\ref{s:hg} for a definition), connecting the curve $(xaby^2)^{-1}$ to $xy$ and to $xaby$. By Lemma~\ref{Lem:addhandle}, this is still taut, and still fails to be a rational homology product.

Notice no such example can exist in genus three, due to the odd genus; any odd-genus sutured handlebody must have a suture set comprised of an even number of curves.
\end{remark}


%

\section{Restricting to one-dimensional representations}\label{s:1d}

The following Proposition is a straightforward generalization of Proposition~5.2 of \cite{AD}, in which $g=n=2$, and may be proved with an analogous argument. We take $M$ to be a balanced sutured handlebody of genus $g$, with $R_+$ connected. Then $\pi_1(M)$ and $\pi_1(R_+)$ are free groups of rank $g$.

Let $\pi_1(M)=\langle x_1,\ldots,x_g\rangle$ and $\pi_1(R_+)=\langle a_1,\ldots, a_g \rangle$, and let $i_*:\pi_1(R_+)\to \pi_1(M)$ be the map induced by the inclusion $i: R_+\into M$. Given a word $w\in \pi_1(M)$, we write $\partial_{x_i}w$ for its Fox derivatives in $\Z[x_1,\ldots, x_g]$ (\cite{Fox}). Notice that any representation $\alpha:\pi_1M \to \GL(V)$ extends naturally to a ring homomorphism $\alpha: \Z[x_1,\ldots, x_g] \to \tn{End}(V)$.

\begin{proposition}\label{Prop:tautness}
For a fixed representation $\alpha:\pi_1(M)\to \GL(V)$, with $\dim V=n$, if the sutured handlebody $M$ is an $\alpha$-homology product, then the $gn\times gn$ matrix
\[
\bigg(\alpha\big(\partial_{x_i}i_*(a_j)\big)\bigg)_{i,j}
\]
has nonzero determinant.

Furthermore, when $\alpha$ is homologically self-dual, this condition is sufficient.
\end{proposition}

The condition of this determinant being nonzero corresponds exactly to $H^2(M, R_+; E_\alpha)$ (and therefore $H_1(M,R_-;E_\alpha)$) vanishing. In the case $\alpha$ is not homologically self-dual, we can still verify tautness by checking that neither this determinant nor that associated to $\alpha^*$ vanishes.

\begin{remark}
From the perspective of the representation variety $\Hom(\pi_1(M),\GL(V))$, the condition given by Proposition~\ref{Prop:tautness} determines a Zariski open subspace of certifying representations. In particular, if a given representation $\alpha$ realizes $M$ as an $\alpha$-homology product, in fact almost every choice of representation $\beta: \pi_1(M) \to \GL(V)$ will also work. Moreover, this gives an indication of how we should expect the situation to differ when we consider $M$ not a handlebody, in that, within the associated representation variety, the set of certifying representations is either empty or the complement of a collection of lower dimensional subvarieties.
\end{remark}

\begin{example}[Theorem~\ref{Thm:nonhom}]
Consider $(M,\gamma)$ from Example~\ref{ex:gen2} and let $\alpha:\pi_1(M)\to\GL_1(\C)$ be any one-dimensional representation. By Proposition~\ref{Prop:tautness}, $M$ is an $\alpha$-homology product when
\[
{\setstretch{1.5}
\det
\begin{pmatrix}
	\alpha\big(\partial_x(xy)\big)	&	\alpha\big(\partial_y(xy)\big)	\\
	\alpha\big(\partial_x(yx)\big)	&	\alpha\big(\partial_y(yx)\big)
\end{pmatrix}
\ne 0.
}
\]
That is to say,
\[
{\setstretch{1.5}
\det
\begin{pmatrix}
	\alpha\big(\partial_x(xy)\big)	&	\alpha\big(\partial_y(xy)\big)	\\
	\alpha\big(\partial_x(yx)\big)	&	\alpha\big(\partial_y(yx)\big)
\end{pmatrix}
=
\det
\begin{pmatrix}
	\alpha(1)	&	\alpha(x)	\\
	\alpha(y)	&	\alpha(1)
\end{pmatrix}
}
=1-\alpha(xy)\ne 0.
\]
Rephrasing what we saw in Example~\ref{ex:gen2}, this shows in particular we cannot take $\alpha$ to be the trivial representation, where $\alpha(x)=\alpha(y)=1$. However, for a generic choice of $\alpha$, this will not be $0$.
\end{example}

Returning to the setting of a genus-$g$ sutured handlebody $M$, consider the case of a one-dimensional representation $\alpha: \pi_1(M)\to \GL_1(\C)$. Here, we have an algebraic understanding of what it means to be a twisted homology product. For a word $w\in \pi_1(M)$, we write $\partial w$ for the vector of Fox derivatives of $w$ with respect to $x_1,\ldots, x_g$.

\begin{proposition}\label{Prop:1Drep}
$M$ is a one-dimensional twisted homology product if and only if the vectors of abelianized Fox derivatives $\tn{ab}(\partial i_*(a_j))$ are linearly independent.
\end{proposition}


\begin{proof}
For the first equivalence, note that for $\alpha$ to satisfy this condition in Proposition~\ref{Prop:tautness}, we require linear independence of the vectors $\alpha(\partial i_*(a_j))$. Consider the composition of maps
\[
\pi_1(M) \xrightarrow{\partial} \Z[\pi_1(M)] \xrightarrow{\alpha} (\GL_1(\C))^g.
\]
Since $\GL_1(\C)$ is abelian, $\alpha$ factors through the abelianization
\[
\pi_1(M) \xrightarrow{\partial} \Z[\pi_1(M)] \xrightarrow{ab} \Z[\Z^g] \xrightarrow{\alpha} (\GL_1(\C))^g.
\]

Thus in this context, a representation $\alpha$ as in Proposition~\ref{Prop:tautness} exists exactly when the vectors $\tn{ab}(\partial i_*(a_j))$ are linearly independent.
%
\end{proof}

We end this section with a lemma which provides a condition for being a one-dimensional twisted product. It will prove useful for finding non-examples. Recall the {\em derived series} $G^{(k)}$ of $G$ is defined by $G^{(0)}=G$ and $G^{(k+1)}=[G^{(k)},G^{(k)}]$.

\begin{lemma}\label{Lem:1Drep}
If $M$ is a one-dimensional twisted homology product, then $\pi_1(R_\pm)\cap \pi_1(M)^{(2)}\subseteq \pi_1(R_\pm)^{(1)}$.
\end{lemma}

\begin{proof}
Suppose $M$ is an $\alpha$-homology product for some $\alpha:\pi_1(M)\to \GL_1(\C)$.

Recall that $H^1(M; E_\alpha)$ is the group of all twisted homomorphisms $f: \pi_1(M)\to \C$, modulo twisted homomorphisms of the form $\hat{z}(g)=\alpha(g)\cdot z - z$ for $z\in \C$. Any $f\in H^1(M; E_\alpha)$ necessarily vanishes on $\pi_1(M)^{(2)}$. We see this first by observing that 
\begin{align*}
f([u,v])	&=	f(u)+\alpha(u) f(v) + \alpha(uv) f(u^{-1}) - \alpha(uvu^{-1}) f(v^{-1})	\\
		&=	f(u)+\alpha(u) f(v) - \alpha(uvu^{-1}) f(u) - \alpha(uvu^{-1}v^{-1}) f(v)	\\
		&=	(1-\alpha(v)) f(u) - (1-\alpha(u)) f(v),
\end{align*}
since $\GL_1(\C)$ is abelian. Now, this is zero when $\alpha(u)=\alpha(v)=1$, for instance, for $u,v\in \pi_1(M)^{(1)}$. Such elements $[u,v]$ normally generate $\pi_1(M)^{(2)}$, so $f$ must vanish on all of $\pi_1(M)^{(2)}$.

Consider now $H^1(R_\pm; E_\alpha)$. Any twisted homomorphism is determined by its values on the generators $a_1,\ldots, a_g$ of $\pi_1(R_\pm)$. Fix $w\in \pi_1(R_\pm)\cap \pi_1(M)^{(2)}$ and let $\#_{a_i} w$ denote the number of occurrences of $a_i$ (counted with sign) in $w$. Notice $\#_{a_i} w = 0$ for all $i$ is exactly the condition for $w\in \pi_1(R_\pm)^{(1)}$. Supposing $w\not\in \pi_1(R_\pm)^{(1)}$, then some $\#_{a_i} w \ne 0$. Define $g\in H^1(R_\pm; E_\alpha)$ by $g(a_j)=\delta_{ij}$. By construction, $g(w)\ne 0$.

Consider the long exact sequence of cohomology groups
\[
\cdots \to H^1(M; E_\alpha)\xrightarrow{i_*} H^1(R_\pm; E_\alpha) \xrightarrow{\delta} H^2(M, R_\pm; E_\alpha) \to \cdots.
\]
As any $f\in H^1(M;E_\alpha)$ vanishes on $w$, the twisted homomorphism $g$ constructed above does not lie in the image of $i_*$. But by exactness, $g$ then is not in the kernel of $\delta$, so $H^2(M, R_\pm; E_\alpha)\ne 0$. By Poincar\'e duality, then $H_1(M, R_\mp; E_\alpha)\ne 0$, which contradicts our assumption that $M$ is an $\alpha$-homology product.
\end{proof}

\section{Higher genus examples which are not one-dimensional homology products}\label{s:hg}

In this section, we consider sutured handlebodies of genus at least three, and see these are inherently more complicated than genus two. We explicitly construct an example of a genus-three handlebody which is not a one-dimensional homology product, and describe how to modify this to produce examples in all higher genus.

For ease of notation, we treat $\gamma$ as a collection of annuli instead of curves. Given a sutured manifold $M$, we can construct a new sutured manifold $N$ by attaching a {\em sutured one-handle}. The one-handle $D^2\times D^1$ is given a product sutured structure $I\times (D^1\times D^1)$. It is attached to $M$ along the disks $I\times (D^1 \times \partial D^1)$, which we require to meet $\gamma$ in two strips so that $0\times (D^1\times \partial D^1)\subset R_-$ and $1\times (D^1\times \partial D^1)\subset R_+$. This construction is illustrated in Figure~\ref{Fig:suthandle}.

\begin{figure}[t!]
  \includegraphics[width=350pt]{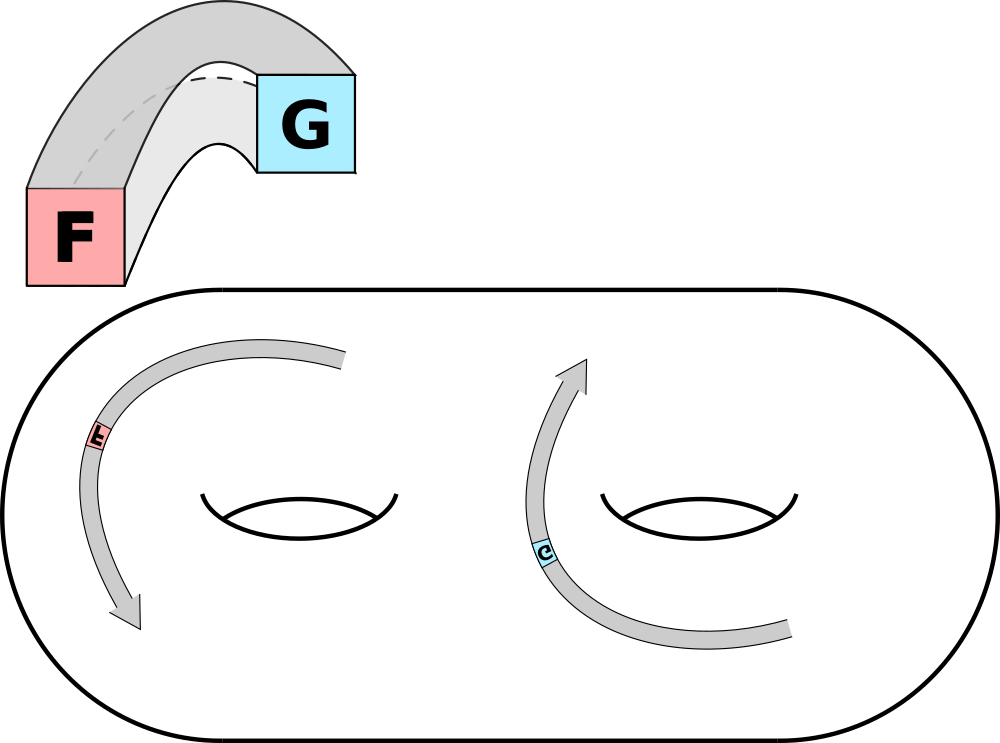}
  \caption{The grey bands represent portions of the suture $\gamma$, viewed as annuli and oriented by the arrows drawn. Homeomorphisms $F$ and $G$ identify the corresponding disks. The letters $F$ and $G$ indicate the maps' respective orientations and restrictions to the boundary.}
  \label{Fig:suthandle}
\end{figure}


\begin{lemma}\label{Lem:addhandle}
Suppose $M$ is a taut sutured manifold. If $(N, R'_\pm, \gamma')$ is obtained by attaching a sutured one-handle to $M$, then $N$ is also taut. Moreover, for any representation $\alpha:\pi_1(M)\to {\tn GL}(V)$, there is a representation $\alpha':\pi_1(N)\to {\tn GL}(V)$ with $\alpha'\vert_{\pi_1(M)}=\alpha$ such that $M$ is an $\alpha$-homology product if and only if $N$ is an $\alpha'$-homology product.
\end{lemma}

\begin{proof}
Note $\pi_1 (N)=\pi_1(M)*\langle x\rangle$, where $x$ is the core of the one-handle. Moreover, $\pi_1(R'_\pm)=\pi_1(R_\pm)*\langle x\rangle$. Define $\alpha'$ to agree with $\alpha$ on $\pi_1(M)$ and to map $x$ to the identity.

Notice
\[
\det \bigg(\alpha'\big(\partial_{x'_i}i_*(a'_j)\big)\bigg) =
\det \begin{pmatrix}
\alpha\big(\partial_{x_i}i_*(a_j)\big)	&	0\\
0							&	I
\end{pmatrix}
=
\det \bigg(\alpha\big(\partial_{x_i}i_*(a_j)\big)\bigg).
\]
The corresponding equality also holds for the dual representations. Thus the result follows from Proposition~\ref{Prop:tautness}.
\end{proof}


\begin{theorem}[Theorem~\ref{Thm:hg}]\label{Thm:ab}
For every $g\ge 3$, there is a taut sutured handlebody $M_g$ of genus $g$ such that $M$ is not an $\alpha$-homology product for any representation $\alpha: \pi_1 (M_g) \to \GL_1\C$.
\end{theorem}

\begin{proof}
Lemma~\ref{Lem:1Drep} shows it suffices to produce a taut example $(M_g, R_\pm,\gamma)$ with $R_+$ containing a curve whose image in $\pi_1(M_g)$ lies in $\pi_1(M_g)^{(2)}$. We will first construct such an example for $g=3$. Let $M$ be a genus-$3$ handlebody, with $\pi_1(M)=\langle x,y,z\rangle$.

\begin{figure}
  \centering
  \begin{subfigure}[b]{0.475\linewidth}
    \includegraphics[width=\linewidth]{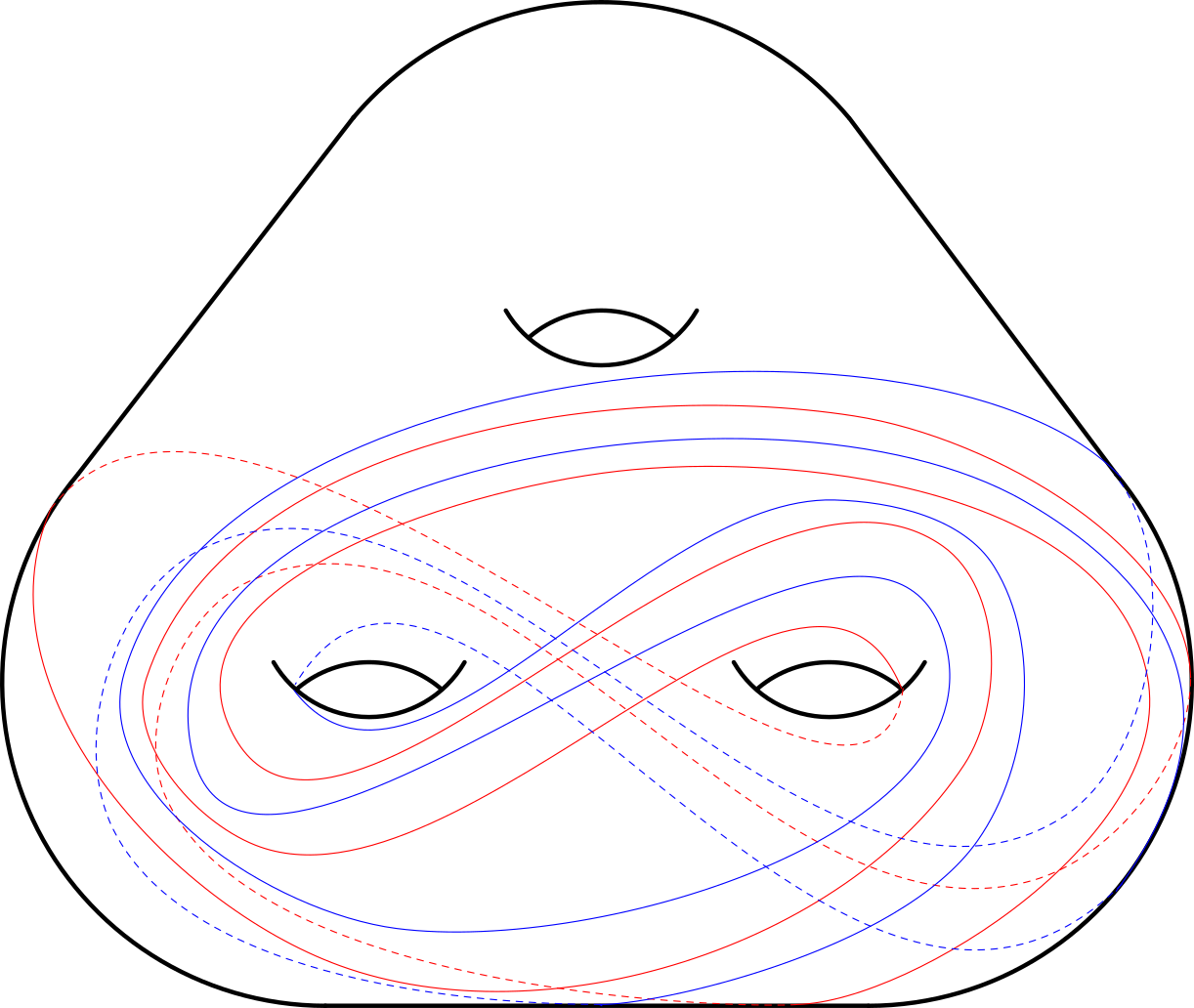}
    \caption{Disjoint curves $A$, $B$ with image in $\pi_1(M)^{(1)}$.}
  \end{subfigure}\hspace{5mm}
  \begin{subfigure}[b]{0.475\linewidth}
    \includegraphics[width=\linewidth]{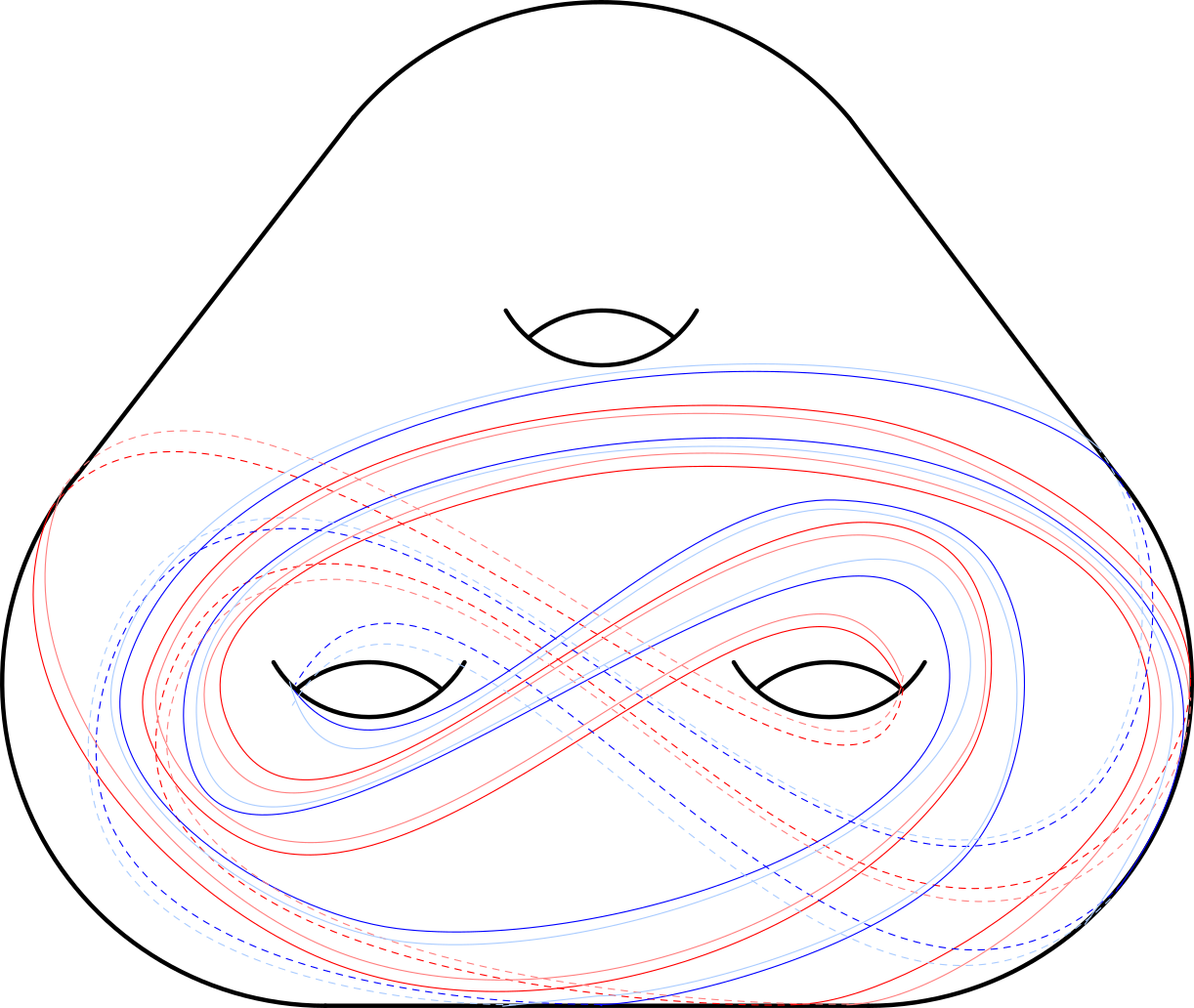}
    \caption{$A$, $B$ doubled and pushed off themselves.}\label{Fig:g3exb}
  \end{subfigure}
  \begin{subfigure}[b]{0.475\linewidth}\vspace{5mm}
    \includegraphics[width=\linewidth]{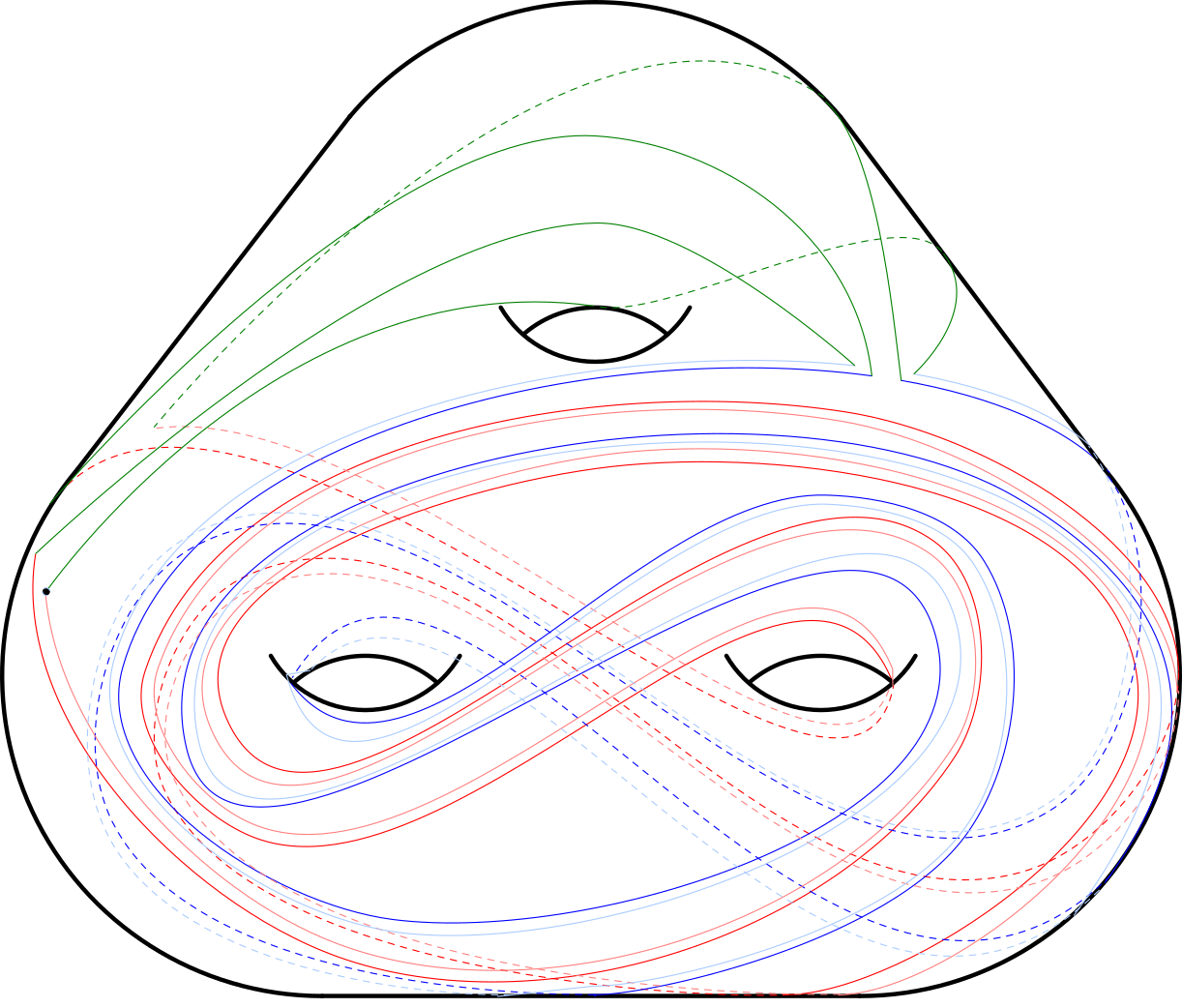}
    \caption{Copies of $A$ and $B$ connected by segments $z$.}\label{Fig:g3exc}
  \end{subfigure}\hspace{5mm}
  \begin{subfigure}[b]{0.475\linewidth}
    \includegraphics[width=\linewidth]{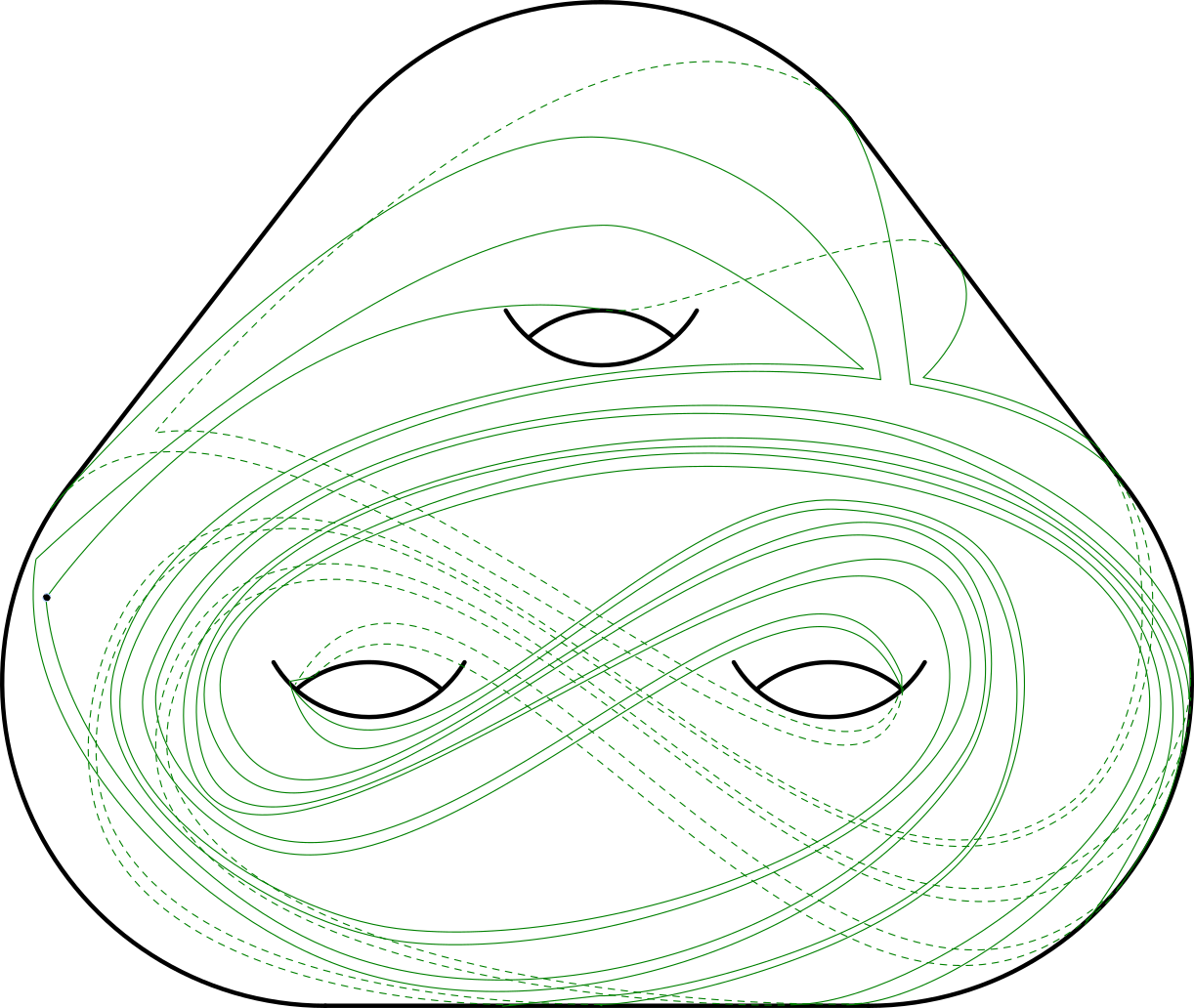}
    \caption{The curve $a=[A, z B z^{-1}]$.}
  \end{subfigure}
  \caption{Construction of the curve $a$.}
  \label{Fig:g3ex}
\end{figure}

\begin{figure}
  \includegraphics[width=.85\linewidth]{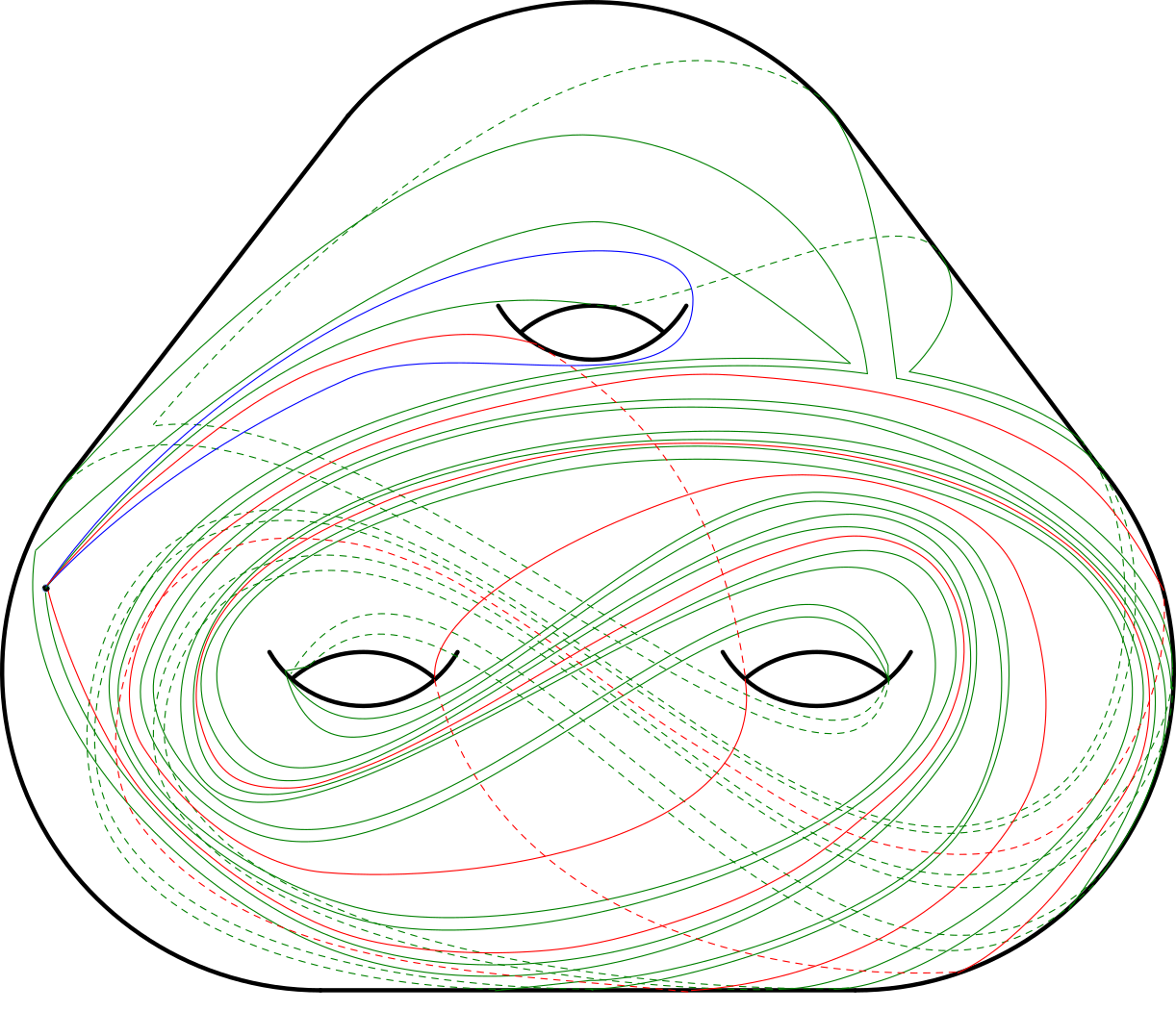}
  \caption{Curves $a$, $b$, and $c$.}
  \label{Fig:g3ex2}
\end{figure}

To describe the sutured structure on $M$, we begin by constructing a simple closed curve $a$ on the boundary of $M$ which lives in $\pi_1(M)^{(2)}$. Figure~\ref{Fig:g3ex} illustrates this process. First, we draw the curves $A$ and $B$, which are disjoint and have image in $\pi_1(M)^{(1)}$. The curve $a$ is constructed from $A$ and $B$ to have image $a=[A, z B z^{-1}]\in \pi_1(M)$. Figures~\ref{Fig:g3exb} and \ref{Fig:g3exc} show this construction, by first taking two copies of each $A$ and $B$, and then connecting them via arcs to yield a simple closed curve with the desired image. Picking a basepoint along $a$, we then find two more simple closed curves $b$ and $c$ on $\partial M$, disjoint away from the basepoint, as shown in Figure~\ref{Fig:g3ex2}. This captures all the information we need to define $(M, R_\pm, \gamma)$: a neighborhood of this defines $R_+$, which is homeomorphic to $\Sigma_{1,2}$, its boundary $\gamma$, and its complement $R_-$. From the construction, we see
\[
\tn{Im}(\pi_1(R_+))=\langle \left[[x,y][x^{-1},y],z[y^{-1},x][y,x]z^{-1}\right], [x,y][y^{-1},x^{-1}], z\rangle.
\]

We now check our example is taut. We do this by exhibiting a two-dimensional representation $\beta: \pi_1(M) \to \GL_2(\C)$ which realizes $M$ as a twisted homology product. Define $\beta$ as follows:
\[
\beta(x)=
\begin{pmatrix}
	1	&	1	\\
	0	&	1	
\end{pmatrix},
\qquad
\beta(y)=
\begin{pmatrix}
	0	&	1	\\
	-1	&	0
\end{pmatrix},
\qquad
\beta(z)=
\begin{pmatrix}
	1	&	0	\\
	0	&	1
\end{pmatrix}.
\]
In fact $\beta$ is a representation to $U(2)$, so the associated twisted homology is self-dual. Then we can apply Proposition~\ref{Prop:tautness}. The relevant matrix is
\[
{\setstretch{2}
\begin{pmatrix}
	\beta\big(\partial_x i_*(a)\big)	&	\beta\big(\partial_y i_*(a)\big)	&	\beta\big(\partial_z i_*(a)\big)	\\
	\beta\big(\partial_x i_*(b)\big)	&	\beta\big(\partial_y i_*(b)\big)	&	\beta\big(\partial_z i_*(b)\big)	\\
	\beta\big(\partial_x i_*(c)\big)	&	\beta\big(\partial_y i_*(c)\big)	&	\beta\big(\partial_z i_*(c)\big)	
\end{pmatrix}}
=
\begin{pmatrix}
	-18	&	-12	&	48	&	-1	&	13	&	-16	\\
	29	&	19	&	-76	&	2	&	-20	&	25	\\
	0	&	-2	&	-4	&	3	&	0	&	0	\\
	1	&	3	&	4	&	-4	&	0	&	0	\\
	0	&	0	&	0	&	0	&	1	&	0	\\
	0	&	0	&	0	&	0	&	0	&	1
\end{pmatrix},
\]
which is invertible, so $M$ is indeed taut.

We may iteratively apply Lemma~\ref{Lem:addhandle} to construct higher genus handlebodies from this example. The process in the Lemma gives a handlebody $M_g$ for all $g>3$ which is still a two-dimensional twisted homology product. Moreover, the obstruction to finding a tenable one-dimensional representation persists, in that the image of the curve $a$ in $\pi_1 (M_g)$ does not change.
\end{proof}

\section{Restricting to solvable representations}\label{s:solv}

We restrict to the case of solvable representations. A group $G$ is {\em solvable} if its derived series $G^{(k)}=[G^{(k-1)},G^{(k-1)}]$ has finite length. For $G$ solvable, let $K$ denote the length of this series, that is, the smallest index $k$ such that $G^{(k)}=1$. Then $G$ is solvable of {\em degree $K$}. This is equivalent to realizing $G$ as a $K$-fold abelian extension of an abelian group. We say a representation $\alpha: G \to \tn{GL}(V)$ is {\em solvable} if it has solvable image, and similarly define the {\em degree of solvability} of $\alpha$ to be the degree of solvability of its image.

In this section, we prove the following.

\begin{theorem}\label{Thm:solv} 
For any $K$, there is a taut sutured handlebody $M_K$ which fails to be a twisted homology product for any solvable representation of degree less than $K$.
\end{theorem}


Observe that Example~\ref{ex:gen2} and Theorem~\ref{Thm:ab} satisfy this Theorem for $K=1,2$, respectively. In the setting of $\tn{GL}_n(\C)$, Zassenhaus shows for a fixed $n$ any solvable subgroup is of bounded degree of solvability (\cite{zassenhaus}). Let $\varphi(K)$ denote the smallest $n$ for which $\tn{GL}_n(\C)$ admits a solvable subgroup of degree $K$.

\begin{corollary}[Theorem~\ref{Thm:dim}]\label{Cor:solv}
The handlebody $M_K$ is not a twisted homology product for any solvable representation to $\tn{GL}_n(\C)$ for $n<\varphi(K)$.
\end{corollary}

In particular, the conjecture of Agol and Dunfield is false when restricted to the class of solvable representations.





The next lemma captures the connection between solvability of a representation and its behavior with respect to the Fox derivative. 

\begin{lemma}\label{Lem:solv}
If $\alpha: G\to \tn{GL}(V)$ is solvable of degree $K$, then $\alpha(\partial g)=0$ for any $g\in G^{(K+1)}$.
\end{lemma}

\begin{proof}
We show this holds for $g=[g_1,g_2]$ where $g_1,g_2\in G^{(K)}$; as elements of this form generate $G^{(K+1)}$, this suffices. Recall 
\[
\partial g = \partial g_1 + g_1 \partial g_2 - g_1 g_2 g_1^{-1} \partial g_1 - g_1 g_2 g_1^{-1} g_2 ^{-1} \partial g_2.
\]
As $\alpha(g_1)=\alpha(g_2)=1$, thus
\[
\alpha(\partial g)=\alpha(\partial g_1) + \alpha(\partial g_2) - \alpha(\partial g_1) - \alpha(\partial g_2) = 0. \qedhere
\]
\end{proof}

The idea of the proof of Theorem~\ref{Thm:solv} is to construct sutured manifolds which carry curves deeper and deeper in the derived series of the manifold's fundamental group, thereby allowing us to exploit this property of the Fox derivative. The construction of these curves follows the same ``double-then-cut-and-paste'' method we use in the proof of Theorem~\ref{Thm:ab} to build a curve in $(\pi_1(M))^{(2)}$.

\begin{proof}[Proof of \ref{Thm:solv}]
We construct the manifolds $M_K$ by induction on $K$. We make the following assumptions on $M_{K-1}$:
\begin{enumerate}
	\item The suture set $\gamma$ consists of a curve $\gamma$. We realize $R_+$ as a closed neighborhood of $g$ simple closed curves $c_1,\ldots c_g$ disjoint away from a common basepoint;
	\item Some curve $c_i$ has image in $\pi_1(M_{K-1})^{(K-1)}\le\pi_1(M_{K-1})$.
\end{enumerate}
Let $M_1$ and $M_2$ be two copies of $M_{K-1}$, and let $a_1$ and $a_2$ denote the curves from condition (2). As the sutures are single curves, there is some $c_i$ in each with geometric intersection $i(c_i,a_j)=1$; denote these by $b_1$ and $b_2$. We first construct an intermediate handlebody $M'_{K}$, by joining $M_1$ and $M_2$ by a one-handle $H_1=D^2\times D^1$ such that the disks $D^2\times\partial D^1$ are identified with disks disjoint from all the curves $c_i$. Then $\pi_1(M_K)=\pi_1(M_1)*\pi_1(M_2)$. Apply the procedure from the proof of Theorem~\ref{Thm:ab} to $a_1$ and $a_2$, as illustrated in Figure~\ref{Fig:solv1}, to construct a curve $a$ whose image in $\pi_1(M'_K)$ is $[a_1,a_2]$, and therefore lies in $\pi_1(M'_{K})^{(K)}$. We fix a basepoint along an arc of $a$ within $H_2$.

\begin{figure}
  \includegraphics[width=300pt]{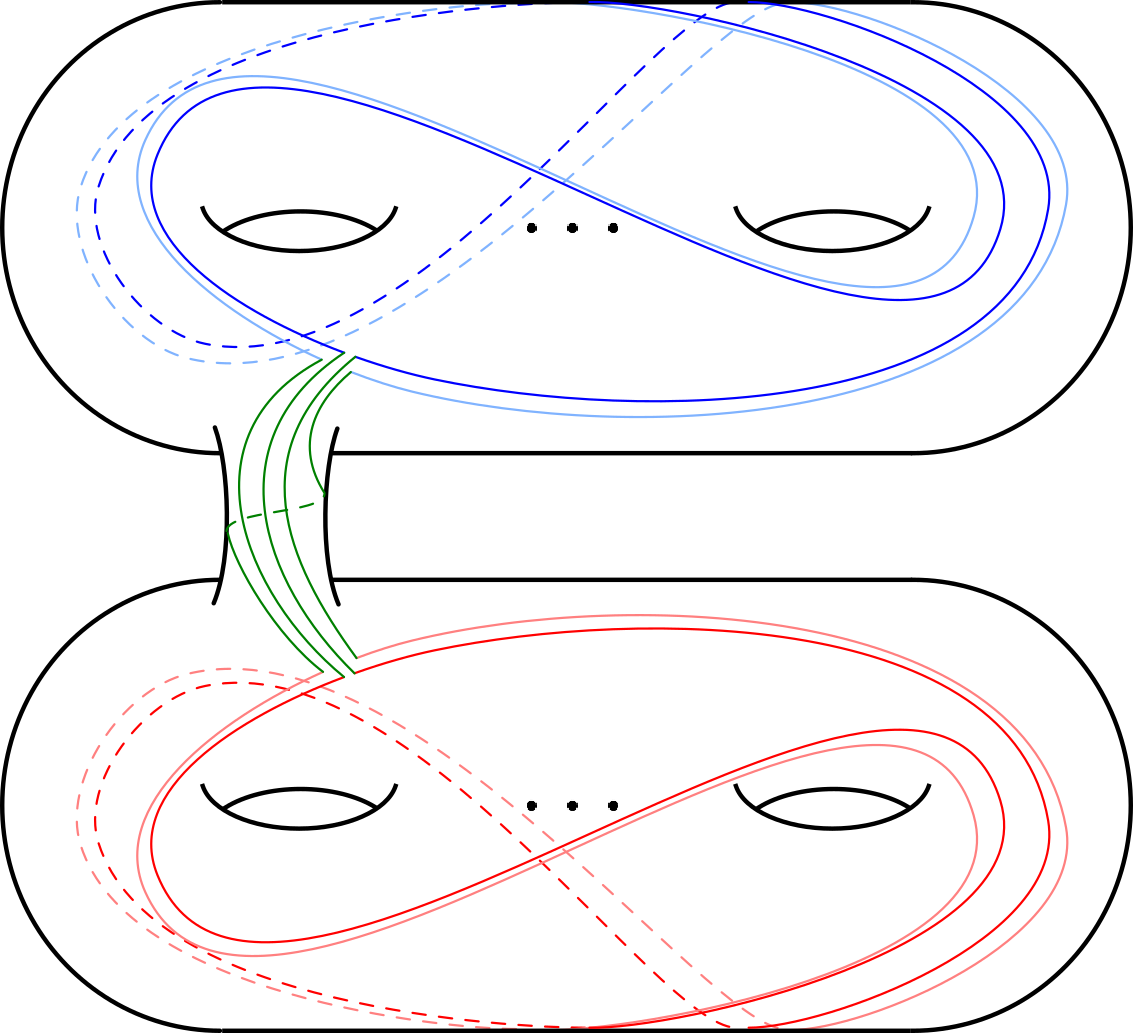}
  \caption{The curve $a$ constructed from $a_1$ and $a_2$.}
  \label{Fig:solv1}
\end{figure}

To obtain $M_K$, we add an additional one-handle $H_2=D^2\times D^1$ to $M'_K$ by attaching the disks $D^2\times\partial D^1$ within a small neighborhood of the basepoint, to either side of the locally separating arc of $a$.

\begin{figure}
  \includegraphics[width=300pt]{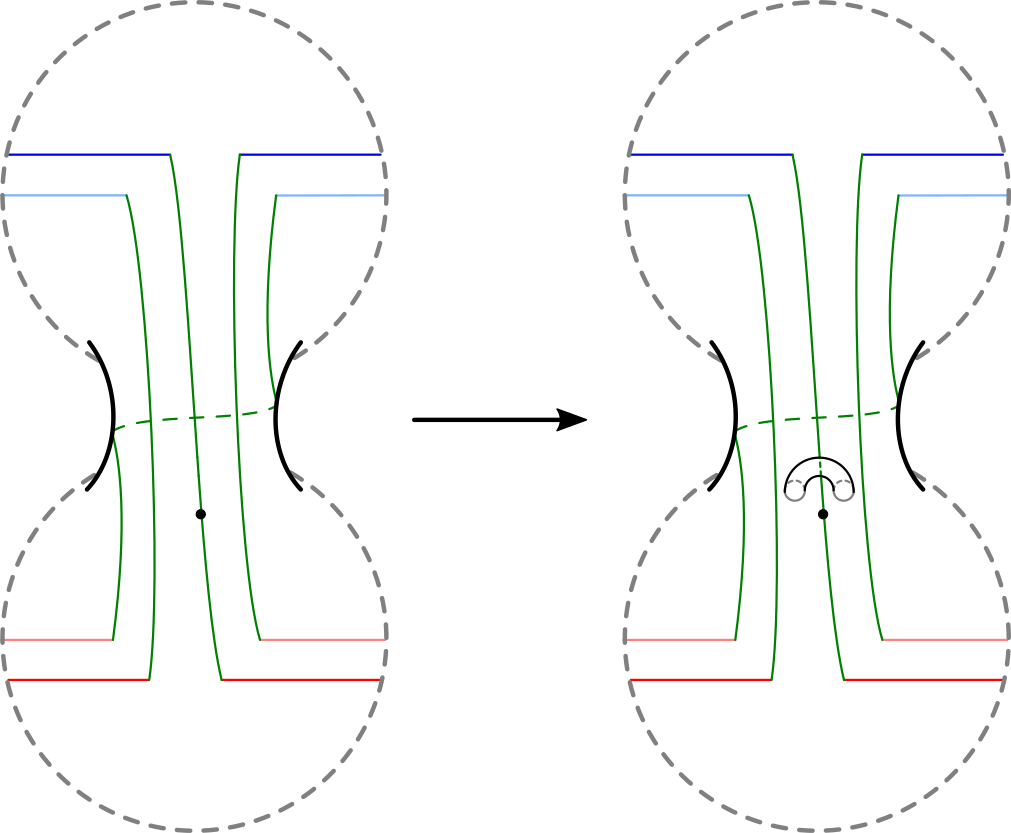}
  \caption{Adding the second one-handle $H_2$ to obtain $M_K$.}
  \label{Fig:solv3}
\end{figure}

To the collection of curves $c_i$ in $\partial M_K$, we add a new curve $c$ which runs around this second handle, parallel to its core, and intersecting $a$ in exactly the basepoint. The remaining curves $c_i$ may intersect $a$. We modify them as illustrated in Figure~\ref{Fig:solv2}. Notice this procedure alters the $\pi_1(M_K)$-image of a curve in one of the following ways:
\begin{itemize}
	\item[] $c_i \mapsto c_i$ \hfill (Figure~\ref{Fig:solv2a})
	\item[] $c_i \mapsto a_j c_i a_j^{-1}$ \hfill (Figure~\ref{Fig:solv2b})
	\item[] $c_i \mapsto c_i a_j^{-1}$ \hfill (Figure~\ref{Fig:solv2c})
\end{itemize}
These curves are once more disjoint away from a basepoint, as Figure~\ref{Fig:solv2d} suggests. While not all combinatorial arrangements of curves are shown, the remaining cases are similar. We add one final curve $b=a_1ca_2$, which is also included in Figure~\ref{Fig:solv2d}, giving a total of $2g+1$ curves. Take a closed neighborhood of these as the new $R_+$ and its boundary as the suture set $\gamma$ defining a sutured structure on $M_{K}$. This construction shows $M_{K}$ satisfies the inductive conditions (1) and (2); in particular the curve $c$ ensures $\gamma$ is connected.

\begin{figure}
  \centering
  \begin{subfigure}[b]{0.475\linewidth}
    \includegraphics[width=\linewidth]{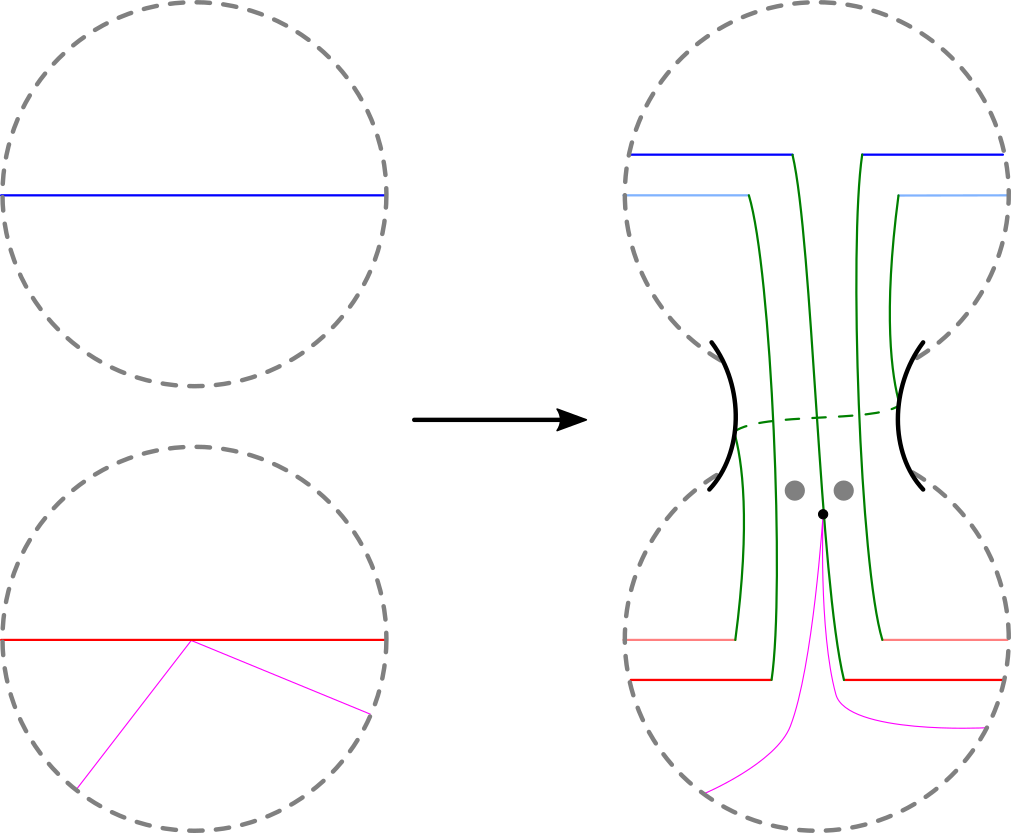}
    \caption{$c_i \mapsto c_i$.}\label{Fig:solv2a}
  \end{subfigure}\hspace{6mm}
  \begin{subfigure}[b]{0.475\linewidth}
    \includegraphics[width=\linewidth]{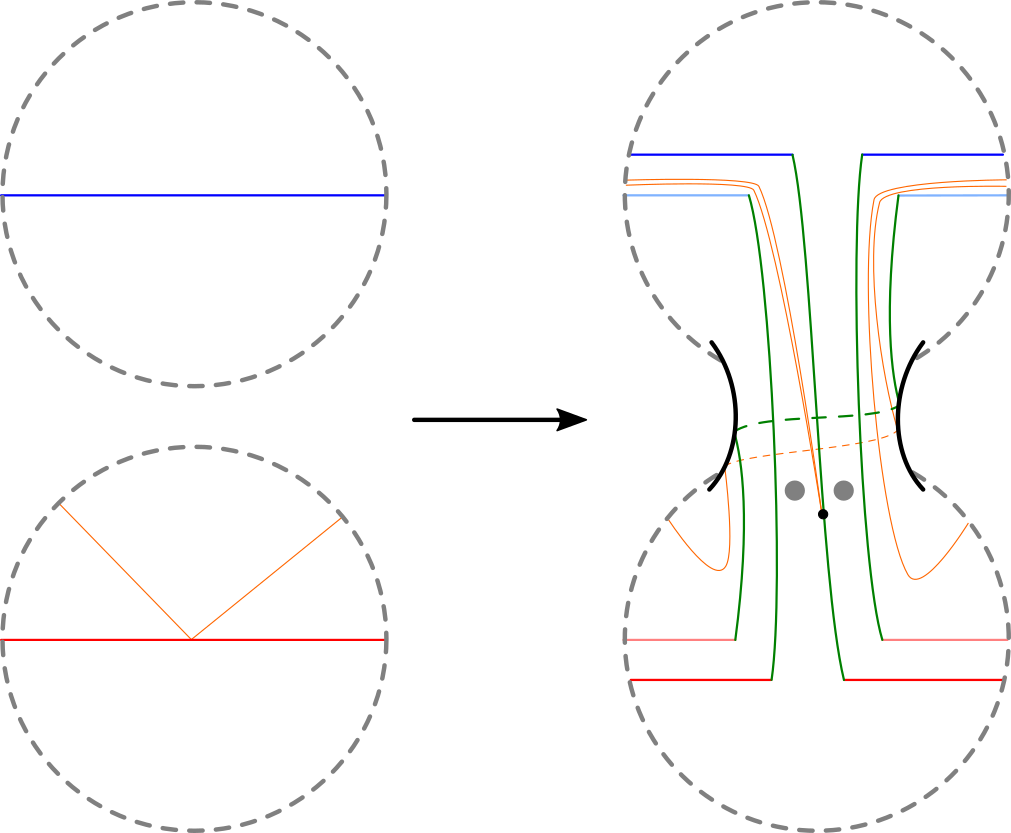}
    \caption{$c_i \mapsto a_j c_i a_j^{-1}$.}\label{Fig:solv2b}
  \end{subfigure}
  \begin{subfigure}[b]{0.475\linewidth}\vspace{5mm}
    \includegraphics[width=\linewidth]{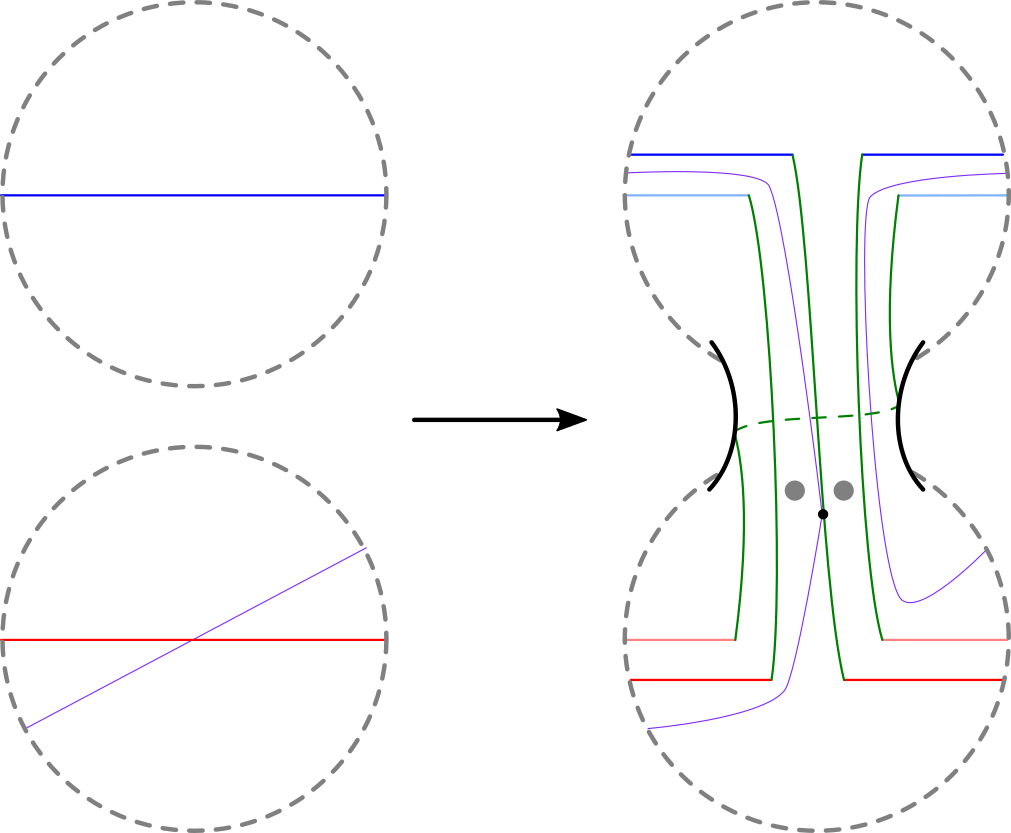}
    \caption{$c_i\mapsto c_i a_j^{-1}$.}\label{Fig:solv2c}
  \end{subfigure}\hspace{6mm}
  \begin{subfigure}[b]{0.475\linewidth}
    \includegraphics[width=\linewidth]{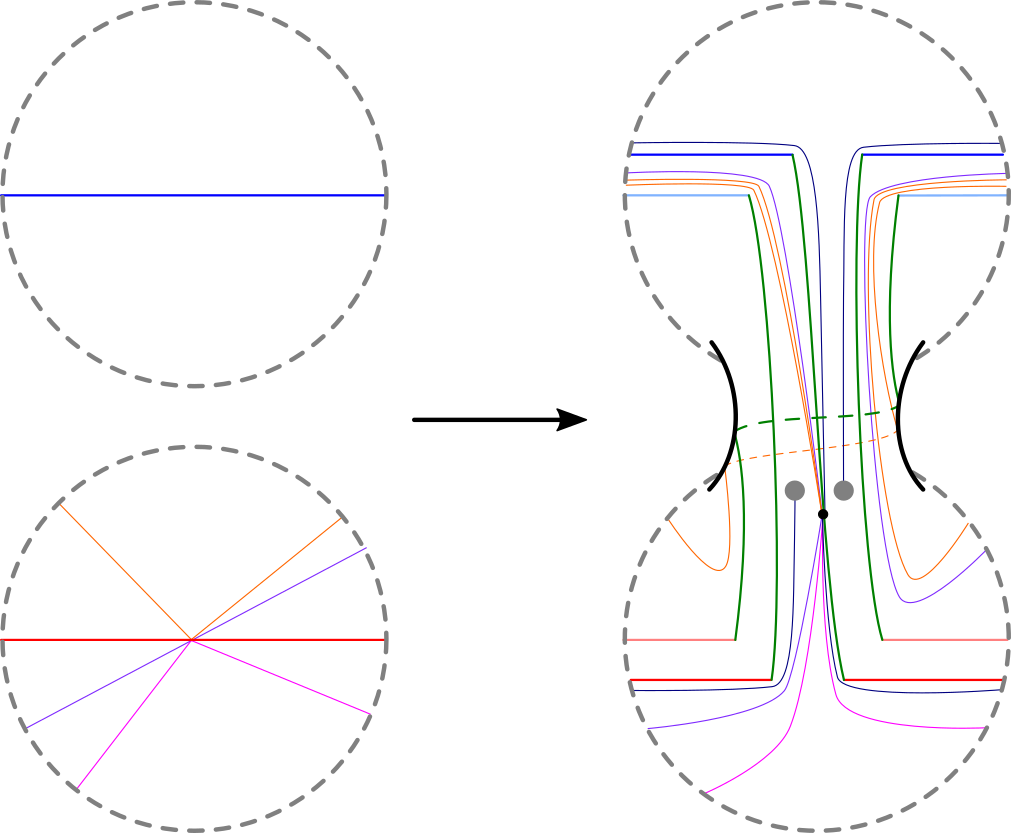}
    \caption{}\label{Fig:solv2d}
  \end{subfigure}
  \caption{Modifying the $c_i$ on $M_K$. The handle $H_2$ is not shown, but is attached at the points shown.}
  \label{Fig:solv2}
\end{figure}

To verify $M_K$ is taut, we exhibit a sutured manifold decomposition
\[
M_K\xrightarrow{S_1} M \xrightarrow{S_2} M' \xrightarrow{S_3} M''\cup M_2,
\]
where $M''$ is another taut sutured handlebody of genus $g$.\footnote{In fact this shows the intermediate manifolds are also taut, in particular $M$, which retains the obstruction to admitting a certifying solvable representation of derived length $K$.} This decomposition is illustrated in Figure~\ref{Fig:decomp2}, and described below.


\begin{figure}
  \centering
  \includegraphics[width=.75\linewidth]{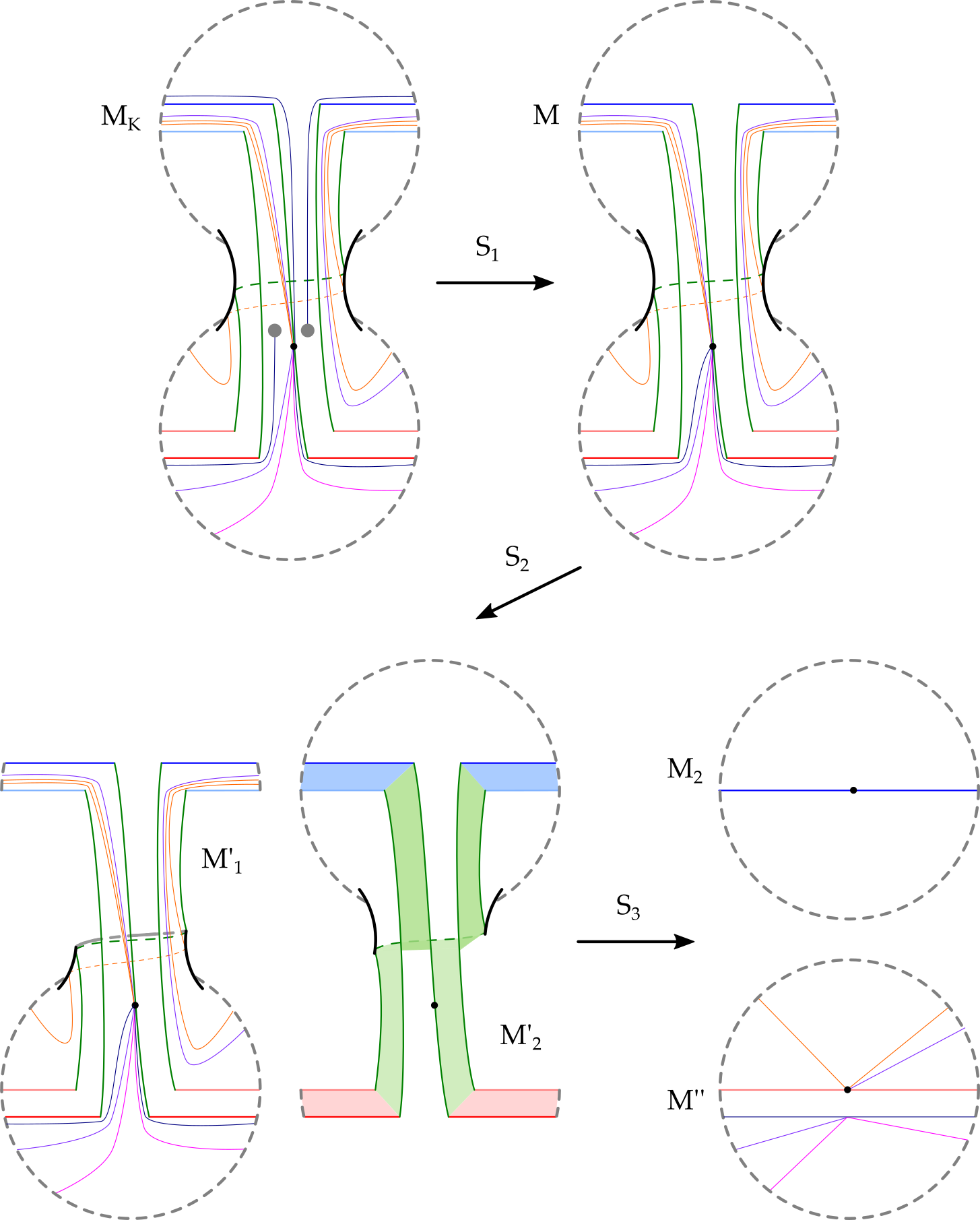}
  \caption{Decomposing $M_K$ into two taut handlebodies of genus $g$.}
  \label{Fig:decomp2}
\end{figure}

The surface $S_1$ is the disk $D^2\times\{\frac{1}{2}\}\subset H_2$. The decomposition kills $c$, and by choosing appropriate choice of orientation of $S_1$, the curve $b=a_1ca_2$ becomes $a_1$.

The surface $S_2$ is a once-punctured torus bound by the curve $a$. Topologically, it is the two strips between the two copies of $a_1$ and $a_2$ used to construct $a$, glued to the disk $D^2\times\{\frac{1}{2}\}\subset H_1$, then pushed slightly into the handlebody. Orient $S_2$ so that $M_2$ lies on the positive side of this disk. This separates $M$ into two genus $g+1$ handlebodies $M'_1$ and $M'_2$. Notice in $M'_2$, the two copies of $a_1$ used to construct $a$ are now parallel in $R_+$, and similarly the copies of $a_2$.

Finally, $S_3$ consists of two disks, each cutting one of the new handles created by the decomposition along $S_2$. Choose these disks to be oriented to agree with $a_1$ and $a_2$, respectively. Additionally, push them off the sutures where possible, to eliminate unnecessary intersections, by dragging the disks toward the basepoint.

In $M'_2$, this results in a disk which intersects the suture in exactly two points, cutting the $a_1$-bands in $R_\pm$. The remainder of the $c_i$ are unaffected, and so the resulting sutured manifold is $M_2$.

In $M'_1$, the situation is more complicated. This decomposition results in a handlebody whose sutured structure is similar to, but not exactly that of $M_1$. The subsurface $R_+$ has fundamental group with generators $c_1,\ldots, c_g$, with the exception of any curve $c_i$ with geometric intersection $i(c_i,a_1)=1$, such as $b_1$. In this case, $b_1$ is replaced by $b_1a_1b_1^{-1}$, and other such $c_i$ can be replaced by $c_ib_1^{-1}$. Notice that the existence of $b_1$ ensures that $R_+$ is connected. We observe, however, that this handlebody is taut exactly when $M_1$ is: on the level of Fox derivatives, this difference translates to
\begin{align*}
\alpha(\partial i_*(b_1))&\mapsto \alpha(\partial i_*(b_1a_1b_1^{-1}))= \alpha(1-i_*(b_1a_1b_1^{-1}))\alpha(\partial i_*(b_1))+\alpha(i_*(b_1))\alpha(\partial i_*(a_1)),\\
\alpha(\partial i_*(c_i)) & \mapsto \alpha(\partial i_*(c_ib_1^{-1}))=\alpha(\partial i_*(c_i))-\alpha(i_*(c_ib_1^{-1}))\alpha(\partial i_*(b_1)).
\end{align*}
In the matrix given by Proposition~\ref{Prop:tautness}, this demonstrates the matrix corresponding to $M''$ is obtained from that for $M_1$ via elementary row operations. This preserves invertibility, unless $\alpha(i_*(b_1a_1b_1^{-1}))=1$; in such a situation $\alpha$ may be perturbed away from this locus, yielding a certifying representation for both $M_1$ and $M''$.

Since $a\in \pi_1(M_K)^{(K)}$, by Lemma~\ref{Lem:solv}, the determinant in Proposition~\ref{Prop:tautness} vanishes for any solvable representation of degree less than $K$. Therefore $M_{K}$ is not a twisted homology product for any such representation.
\end{proof}

\newpage
\bibliographystyle{alpha}
\bibliography{bibel}

\end{document}